\documentclass[11pt]{article}
\usepackage{mymacros}

\title{A criterion on the free energy for log-Sobolev inequalities in mean-field particle systems}

\author{Roland Bauerschmidt\footnote{Courant Institute of Mathematical Sciences, NYU. E-mail: {\tt bauerschmidt@cims.nyu.edu}.}
\and Thierry Bodineau \footnote{I.H.E.S., Universit\'e Paris-Saclay, CNRS, Laboratoire Alexandre Grothendieck. 
35 Route de Chartres, 91440 Bures-sur-Yvette, France. E-mail: {\tt bodineau@ihes.fr}.}
\and Benoit Dagallier\footnote{CEREMADE, Universit\'e Paris-Dauphine, PSL University. E-mail: {\tt dagallier@ceremade.dauphine.fr}.}
}

\begin{document}
\maketitle

\begin{abstract}
  For a class of mean-field particle systems,
  we formulate a criterion in terms of the free energy that implies uniform bounds on the log-Sobolev constant of the associated Langevin dynamics. 
  For certain double-well potentials with quadratic interaction, the criterion 
   holds up to the critical temperature of the model, and
  we also obtain precise asymptotics on the decay of the log-Sobolev constant when approaching the critical point.
  The criterion also applies to ``diluted'' mean-field models defined on sufficiently dense, possibly random graphs.
  We further generalize the criterion to non-quadratic interactions that admit a mode decomposition.
The mode decomposition is different from the scale decomposition of the Polchinski flow we used for short-range spin systems.
\end{abstract}

\section{Introduction}

Let $V : \R^d \to \R$, $W  : \R^d \times \R^d \to \R$ be symmetric, $C^2$ functions, and let $T>0$.  
We are interested in characterising the large time behaviour of the following Langevin mean-field dynamics for large $N$:
\begin{equation}
i \leq N, \qquad 
  dX_t^{i} = -\Big[\nabla V(X_t^{i}) + \frac{1}{NT} \sum_{j=1}^N \nabla_1 W(X_t^{i}, X_t^{j})\Big]\, dt + \sqrt{2}\,  dB_t^i ,
  \label{eq: meanfield_Langevin}
\end{equation}
where $\nabla_1$ denotes gradient with respect to the first coordinate and with the $B_t^i$ independent standard Brownian motions.  
The literature on this question and on the associated McKean--Vlasov equation, obtained as the limiting dynamics of $X_t^1$ as $N\to\infty$,
is extremely vast,   
see e.g.~the surveys \cite{zbMATH04211245,zbMATH07596816,zbMATH07596817} and, in the McKean--Vlasov case, 
the landmark paper~\cite{MR2053570}. 
Below we will exclusively discuss the interacting particle system~\eqref{eq: meanfield_Langevin}. 
We shall only mention the works most relevant to our setting, referring to the above works for additional bibliography. 

\medskip

Under suitable assumptions on the  potentials $V,W$ (referred to below as the confinement respectively the interaction potential),  
the law of the dynamics~\eqref{eq: meanfield_Langevin} converges to a unique invariant measure given by
\begin{equation}
m^N_T(dx) 
=
\frac{1}{Z^N_T}\exp\Big[-\frac{1}{2TN}\sum_{i,j=1}^N W(x_i,x_j)\Big] \prod_{i=1}^N\alpha_V(dx_i),
\label{eq_measure_nonquad}
\end{equation}
where $\alpha_V\in{\bf M}_1(\R^d)$ denotes the absolutely continuous probability measure 
\begin{equation}
\label{eq: def alpha V}
\alpha_V (dx) \propto e^{-V(x)}\, dx,
\end{equation}
and $\propto$ stands for equality up to a normalisation factor. 
Throughout the paper we always implicitly assume that $m^N_T$ is a probability measure:
\begin{equation}
\forall \, T>0,\qquad Z^N_T<\infty
.
\label{eq_integrability_ZNT}
\end{equation}
For large $N$, the behaviour of the dynamics and the measure $m^N_T$ are governed by the free energy $\cF_T(\rho)$, 
defined for an absolutely continuous probability measure $\rho(dx)=\rho(x)\, dx$ on $\R^d$ by:
\begin{equation}
  \cF_T(\rho) = \int_{\R^d} \rho(x) \log \rho(x) \, dx + \int_{\R^d} V(x) \, \rho(dx)\, +\frac{1}{2T} \int_{(\R^d)^2} W(x,y) \, \rho(dx)\, \rho(dy),
  \label{eq_def_free_energy}
\end{equation}
and equal to $+\infty$ if $\rho$ is not absolutely continuous. 
Under general conditions on $V,W$, 
it is known that $\cF_T$ admits at least one minimiser, see~\cite{LIU2020503} for the latest, most general results with references to earlier works. 
Moreover, if $T$ is large enough, 
then $\cF_T$ has a unique minimiser, 
see, e.g., \cite{MR2053570,Guillin2022UniformPA}. 
Conversely, minimisers (local or global) for small enough $T$ are in general not unique~\cite{MR4604897,monmarcheReygner2024localLSI}. 
In statistical mechanics terms the existence of a temperature $T_c\in(0,\infty)$ which separates regions where uniqueness and non-uniqueness hold corresponds to a phase transition:
\begin{equation}
T_c := \inf \Big\{ T>0 : \mathcal F_{T'}\text{ has a unique global minimiser for each } T'> T\Big\} .
\label{eq_def_beta_c}
\end{equation}

We are interested in relating this critical temperature to relaxation properties of the dynamics~\eqref{eq: meanfield_Langevin}. 
The mean-field measure $m^N_T$ of~\eqref{eq_measure_nonquad} is said to satisfy a log-Sobolev inequality with constant $\gamma>0$ if,  for any $C^\infty$ compactly supported $F:(\R^d)^N\to\R_+$,
\begin{equation}
\label{eq: LSI definition}
\ent_{m^N_T}(F)
\leq 
\frac{2}{\gamma} \int |\nabla \sqrt{F}|^2\, dm^N_T,
\end{equation}
with $\ent_{m^N_T}(F)=\E_{m^N_T}[F\log F]-\E_{m^N_T}[F]\log\E_{m^N_T}[F]$. 
Under mild conditions on $V,W$, this inequality holds for an optimal constant $\gamma^N_{\mathrm{LS}}(T) >0$
(see e.g.~\cite{MR1971582} for background on log-Sobolev inequalities in statistical mechanics context).

\medskip

This paper focuses on deriving uniform in $N$ estimates of the log-Sobolev constant from which quantitative   controls on the relaxation of the Langevin dynamics \eqref{eq: meanfield_Langevin} follow.
 This also applies to the limiting McKean--Vlasov equation, see for example the discussion in \cite{Guillin2022UniformPA}.
Our main interest is in bounding the log-Sobolev constant under assumptions that only involve the free energy $\cF_T$.

The question of uniform bounds on the log-Sobolev constant has already received a lot of attention.  
The equilibrium phase transitions are determined at the macroscopic level by the mean-field functional $\cF_T$ \eqref{eq_def_free_energy} which records the contribution of the interaction and the entropy of the system. 
In general, 
non-convexity of the interaction $\rho\mapsto \int W\rho^{\otimes 2}$ may create a phase transition depending on the temperature $T$. 
In this case, the log-Sobolev constant will vanish with $N$. 
In fact, even in absence of a phase transition, the existence of local minima for $\cF_T$ will lead to a metastable behaviour of the dynamics and the log-Sobolev constant in \eqref{eq: LSI definition} is expected to also vanish with $N$,   
a fact established quite generally in~\cite{MR4604897}.  
At sufficiently high temperature/small interaction one expects neither phase transition nor metastability and therefore uniform bounds on the log-Sobolev constants should hold, 
as was shown for a large class of $V,W$ in \cite{Guillin2022UniformPA}. 
Very recently, the log-Sobolev inequality was derived for possibly large but flat convex interactions (see~\eqref{eq: def flat convex}) in  \cite{wang2024LSI,chewi2024LSI}. 
Building on the result of \cite{wang2024LSI} perturbations to the flat convex case were studied in  \cite{monmarche2024LSI}. 
In this work, under various assumptions on the confinement and interaction potentials $V, W$ and on the temperature parameter $T>0$, we are going to  relate the scaling of the log-Sobolev constant \eqref{eq: LSI definition} to conditions on the mean-field functional $\cF_T$ \eqref{eq_def_free_energy}.

Our approach closely follows the strategy introduced in \cite{MR3926125} based on renormalisation ideas, 
exposed in much greater generality in the survey~\cite{10.1214/24-PS27}. 
The method of~\cite{MR3926125} was  applied in \cite{zbMATH07286868} to analyse in depth the behaviour of the spectral gap for discrete mean-field models and its precise divergence in $N$ close to critical regimes.
Here we use it to study the mean-field model~\eqref{eq_measure_nonquad} defined in the continuum, 
but also models beyond the strict (fully connected) mean-field setting. 
Indeed, the strategy in \cite{MR3926125,10.1214/24-PS27} is well suited to studying models with general, possibly random interactions as it relies only on the spectral structure of the interaction matrix. 
It has for instance been extended in \cite{Bauerschmidt:2024aa} to study Kawasaki dynamics of the Ising model on random regular graphs and we extend it here to continuous models on random graphs with sufficiently large degrees.

For quadratic interactions, either of mean-field type or on suitable random graphs, 
we will show that for a large class of models  the log-Sobolev inequality can be analysed up to the critical point, 
characterised as above in terms of the free energy functional,
see Theorem~\ref{thm: quadratic interaction} in the quadratic case and Theorem~\ref{thm: quadratic interaction random} on random graphs. 
For this, we use a spectral decomposition of the interaction matrix to reduce, in the large $N$ limit, the complexity of the microscopic dynamics  to the analysis of a single (slow) mode which determines the macroscopic behaviour.
This spectral decomposition is trivial for fully connected models and otherwise relies on expander properties of random graphs with large degrees.
Similar strategies based on spectral decomposition of the measure have recently been employed to great success for discrete models, 
see e.g.~
\cite{MR4408509,10.1214/18-EJP159} and most recently in~\cite{Anari2024TrickleDownIL,
Alaoui2023FastRO,
liu2024fastmixingsparserandom}.

For non-quadratic interactions, the dimensional reduction is less straightforward and we focus on fully connected graphs, 
i.e., on the mean-field measure~\eqref{eq_measure_nonquad}. 
For flat convex interactions, 
there is no phase transition as  the mean-field functional $\cF_T$ \eqref{eq_def_free_energy} is strictly convex at all temperatures. 
In this case the log-Sobolev inequality was derived  in  \cite{wang2024LSI,chewi2024LSI}.
We consider a specific kind of non-convex interactions for which the mode decomposition used in the quadratic interaction case can be generalised. 
By projecting the  mean-field functional $\cF_T$ on the modes, we obtain a 
criterion involving only $\cF_T$ which implies that the log-Sobolev inequality holds uniformly in $N$.

In specific instances, using detailed features of these models,
  we were previously able to analyse short-range spin and field theory models, such as continuum limits which arise
  as invariant measures of singular SPDEs or critical Ising models in $d\geq 5$, see
  \cite{MR4061408,MR4303014,MR4720217,MR4705299} and \cite{10.1214/24-PS27} for an introduction.
  An essential feature in the analysis of these models is a \emph{scale decomposition} in terms of the Polchinski flow \cite{10.1214/24-PS27}.
  For mean-field particle systems, the perspective is different.
  In the interpretation of the particle system as a spin system (with spins taking values in $\R^d$ corresponding to particle positions),
  there is essentially only a single scale (the mean field).
  On the other hand, possibly more complicated particle interaction is captured by the structure of the interaction potential
  which can now have different modes (compared to the spins systems which have short-range, but usually quadratic interaction potentials).
  This requires a \emph{mode decomposition} (instead of scale decomposition)
  that we explore in the mean-field setting in this paper.
  For quadratic interactions there will only be a single mode.

\bigskip

The main results of this paper are stated in the next subsections as well as more specific references depending on the structure of the interactions. 
To avoid technical issues, we restrict to the following class of confinement potentials:
\begin{assumption}[Assumptions on $V$]
\label{ass_Vnonquad}
The potential $V\in C^2(\R^d,\R)$ can be decomposed as $V=V_c + \tilde V$, 
where $V_c\in C^2(\R^d,\R)$ satisfies $\He V_c\geq \id$ and where  $\tilde V$ is Lipschitz or bounded.
\end{assumption}

\subsection{Quadratic interactions}
\label{sec: Quadratic interactions intro}

We first consider the simplest case of a quadratic interaction of the form $W(x,y) = - (x,y)$ which leads to a single mode (and a single scale).
The continuous Curie--Weiss model (see~\eqref{eq: potentiel x4-x2} below) is the prototypical example in this class. 
The results in this case generalise the method of \cite{MR3926125,10.1214/24-PS27} (in the mean-field case)
and prepare for the developments of Sections~\ref{sec: quadratic diluted} and \ref{sec: general interactions} by providing a new perspective that focuses on the mean-field free energy functional.

For $m \in \R^d$, define the (one mode) coarse grained free energy as: 
\begin{equation}
\label{eq: coarse grained m}
\hat \cF_T(m) =
\inf \left\{ \cF_T(\rho), \quad \rho \ \text{such that} \int x\, \rho(dx)=m \right\} .
\end{equation}
Our goal is to relate the  log-Sobolev inequality of the mean-field measure $m^N_T$ \eqref{eq_measure_nonquad} to properties of the free energy $\hat \cF_T$,
i.e.~to macroscopic properties of the system.
In particular, for $T>T_c$, we are going to assume that the gradient flow associated with  $\hat \cF_T$,
\begin{equation}
\label{eq: gradient flow m}
\dot m_t = - \nabla \hat \cF_T(m), \quad    m_0 \in\R^d,
\end{equation}
relaxes exponentially fast to the global minimum $m^\star$, 
 i.e., $\cF(m_t)-\cF(m^\star) \leq e^{-2\gamma t}(\cF(m_0)-\cF(m^\star))$.
It is known (see \cite{karimi2016linear,CheStr24PL} for references) that this exponential relaxation of the dynamics is equivalent to the following  Polyak-\L ojasiewicz inequality with constant $\gamma= \gamma_{\rm PL}$,
\begin{equation}
\label{eq: Polyak-Lojasiewicz hat F}
\hat \cF_T (m) - \hat\cF_T (m^\star) \leq \frac{1}{2\gamma_{\rm PL}} \| \nabla \hat\cF_T (m) \|^2, \quad \forall  m \in\R^d.
\end{equation}
It is implied  by uniform convexity of $\hat \cF_T$  (but more general). 
In addition, it is shown in~\cite[Theorem~1]{CheStr24PL} that  inequality \eqref{eq: Polyak-Lojasiewicz hat F} implies the following log-Sobolev inequality.
\begin{lemma}{(\cite[Theorem 1]{CheStr24PL})}
The Polyak-\L ojasiewicz inequality~\eqref{eq: Polyak-Lojasiewicz hat F} holds with constant $\gamma_{\rm PL}>0$ if and only if the probability measure $\propto e^{-N\hat \cF_T(m)}\, dm$ has a log-Sobolev constant $\gamma_{\rm PL}N(1+o_N(1))$. 
\end{lemma}
The following theorem shows that inequality \eqref{eq: Polyak-Lojasiewicz hat F} implies a log-Sobolev inequality for the mean-field measure $m^N_T$ \eqref{eq_measure_nonquad} uniformly in $N$.

\begin{theorem}[Quadratic interaction]
\label{thm: quadratic interaction}
Let the confinement potential $V$ satisfy Assumption \ref{ass_Vnonquad} and  the interaction be given by:
\begin{equation}
\label{eq: potentiel quadratique}
 W(x,y) = - (x,y) ,\qquad
 x,y  \in \R^d.
\end{equation}
Let $T>T_c$ (defined in \eqref{eq_def_beta_c}) and assume that  $\hat\cF_T$ satisfies a Polyak-\L ojasiewicz inequality \eqref{eq: Polyak-Lojasiewicz hat F}. 
Then the measure $m^N_T$ satisfies a log-Sobolev inequality with a constant independent of $N$.  
\end{theorem}

For a large class of potentials, Theorem \eqref{thm: quadratic interaction} implies a log-Sobolev inequality up to $T_c$: 
\begin{corollary}[Double well confinement potentials]
\label{cor: double well}
Under the same conditions as the previous theorem,
let $d=1$, and consider the choice of confinement potential of the form
\begin{equation}
\label{eq: potentiel x4-x2}
 V(x)
 =
 \frac{x^4}{4} - \lambda\frac{x^2}{2}
 ,\qquad
 x,y,\lambda\in\R,
\end{equation}
or more generally suppose  that $V$ is in the GHS class, i.e.~satisfies Assumption~\ref{ass_doublewell} below.
Then for any $T>T_c$ (defined in \eqref{eq_def_beta_c}), 
the measure $m^N_T$ satisfies a log-Sobolev inequality with a constant independent of $N$.  
Moreover, the log-Sobolev constant vanishes  linearly at $T_c$: 
\begin{equation} 	
\limsup_{N\to\infty} \gamma^N_{\mathrm{LS}}(T)	
\leq 
c_1 (T-T_c)
,\qquad
\gamma^N_{\mathrm{LS}}(T)	\geq 
c_2 (T-T_c)
,\quad 
N\geq 1,
\end{equation}
for some  constants $c_1, c_2>0$.
\end{corollary}
 Note that a similar result would hold for the interaction $W(x,y) = \frac12 (x -y)^2$
as one can recover the structure \eqref{eq: potentiel quadratique} by rewriting the Hamiltonian and changing $V$ as:
\begin{equation}
\frac{1}{4TN} \sum_{i,j=1}^N (x_i - x_j)^2 + \sum_{i=1}^N V(x_i) 
=  -\frac{1}{TN} \sum_{i,j=1}^N x_i x_j  + \sum_{i=1}^N \left( V(x_i) + \frac{x_i^2}{T}  \right).
\end{equation}

\subsection{Quadratic interactions on non-complete graphs}
\label{sec: quadratic diluted}

For many applications, it is natural to consider interactions on general graphs with large degrees but which are not fully connected. In this case, the mean-field theory does not apply and a specific analysis is needed (see e.g.~\cite{zbMATH07692278,zbMATH07768340,ayi-hal-04394768,zbMATH07873666}).
In this section, we consider interactions indexed by the edges of random graphs and extend to this case the  method implemented to prove Theorem~\ref{thm: quadratic interaction}.

\medskip

We first introduce some notation.
Consider a graph $G_N$ on $\{1, \dots , N\}$ with adjacency matrix:
\begin{equation}
A_{ij}
=
{\bf 1}_{i\sim j}
,\quad 
A_{ii}
=
0
,\qquad
i,j\in G_N
,
\end{equation}
where $i\sim j$ means that there is an edge between $i,j$ in $G_N$. 
We consider the following probability measure on $\R^{N}$ with interactions restricted to the edges of the graph $G_N$: 
\begin{equation}
m^{G_N}_T(dx)
=
\frac{1}{Z^{G_N}_T}\exp \Big[\frac{1}{2Td_N}(x,Ax)\Big]
\prod_{i=1}^N \alpha_V(dx_i),
\label{eq: measure_graph}
\end{equation}
where $d_N$ is the average degree of the graph, 
see Theorem~\ref{thm: quadratic interaction random} below for a precise definition of $d_N$.
Write 
$\gamma^{G_N}_{\mathrm{LS}}(T)$ for the   log-Sobolev constant~\eqref{eq: LSI definition}
for the measure~\eqref{eq: measure_graph}.
\begin{theorem}
\label{thm: quadratic interaction random}
Let $V(x)=\frac{x^4}{4}-\lambda \frac{x^2}{2}$ ($\lambda\in\R$) or more generally suppose that $V$ satisfies Assumption~\ref{ass_doublewell}. 
Let $T_c$ denote the critical temperature \eqref{eq_def_beta_c} of the fully connected mean-field-model $m^{N}_T$ with $W (x,y) = - x y$. 
Let $\bbP_N$ be the uniform measure on random regular graphs $G_N$ in $\{1, \dots , N\}$ with fixed degree $d_N$ at each site or the measure of Erd\"os-R\'enyi graphs with mean degree $d_N$.

Assume either that $\lim_{N\to\infty} d_N=\infty$ in the random regular graph case, 
or that $\lim_{N\to\infty} d_N/\log N= \infty$ in the Erd\"os-R\'enyi case.
Then:
\begin{itemize}
	\item[(i)]
For $T>T_c$, there is a constant $\gamma_T>0$ such that
\begin{align}
\label{eq: convergence borne LSI}
\lim_{N \to \infty} \bbP_N \big[  \gamma^{G_N}_{\mathrm{LS}}(T) \geq  \gamma_T \big] = 1.
\end{align}
	\item[(ii)] For $T<T_c$, there is a sequence $(\delta_N)$ converging to 0 such that: 
\begin{align}
\label{eq: convergence borne LSI 2}
\lim_{N \to \infty} \bbP_N \big[  \gamma^{G_N}_{\mathrm{LS}}(T) \leq  \delta_N \big] = 1.
\end{align}
\end{itemize}
\end{theorem}
\begin{remark}
\begin{itemize}
	\item[(i)] The proof of Theorem~\ref{thm: quadratic interaction random} provides a bound of the log-Sobolev constant at fixed $N$, 
	on any graph which is a sufficiently good expander. 
This is a condition that only involves the spectrum of the adjacency matrix (see Assumption~\ref{ass_graph} and Remark~\ref{rmk_more_graphs}).  
The precise distribution of the random graphs is therefore not relevant. 
\item[(ii)] A natural extension of Theorem \ref{thm: quadratic interaction random} would be to analyse the critical behavior on random graphs with large, but finite degrees. This question was addressed in \cite{MR3059200} for discrete models, and  adapted to our framework in \cite[Example 6.19]{10.1214/24-PS27}. 
\item[(iii)] As for Theorem~\ref{thm: quadratic interaction}, 
	a uniform log-Sobolev inequality for $m^{G_N}_T$ holds for more general $V$ satisfying only Assumption~\ref{ass_Vnonquad}. 
	The log-Sobolev inequality is then valid up to a temperature that may be higher than $T_c$. 
\end{itemize}
\end{remark}

\subsection{General interactions}
\label{sec: general interactions}

For general interactions $W$, the validity of the  log-Sobolev inequality is less understood. 
Let $W$ be an interaction with bounded second derivatives.
It was recently shown in \cite{wang2024LSI,chewi2024LSI} that if
$W$ is flat convex, i.e., for each $\rho_1,\rho_2\in {\bf M}_1(\R^d)$, 
\begin{equation}
\label{eq: def flat convex}
u \in[0,1]\mapsto \int W^+(x,y)\rho_u(dx) \; \rho_u(dy) \quad\text{is convex}
\qquad 
\big(\rho_u := (1-u) \rho_1+ u \rho_2 \big) , 
\end{equation}
then a uniform log-Sobolev inequality holds for every $T > 0$. In this case, $\mathcal F_T$ is   convex for any $T$ and there is no phase transition so that the critical value  defined in \eqref{eq_def_beta_c} is such that $T_c= 0$. Note that more general convex interactions than the two body potential $W(x,y)$ are covered by \cite{wang2024LSI,chewi2024LSI}.
In \cite{monmarche2024LSI}, the results \cite{wang2024LSI} were extended beyond the flat convex case.

\medskip

 Suppose that the  confinement potential $V$ satisfies Assumption~\ref{ass_Vnonquad}.
We consider  a  class of interaction potentials $W$ with a non-convex part.
\begin{assumption}[Assumptions on $W$]
\label{ass_Wnonquad}
$W\in C^2(\R^{d}\times\R^d,\R)$ is symmetric and can be decomposed as:
\begin{equation}
W
=
W^+ - W^- ,\qquad
W^\pm\in C^2(\R^d\times\R^d,\R)
,
\end{equation}
where $W^\pm$ are symmetric, 
and:
\begin{itemize}
\item $W^+$ is bounded and flat convex (see \eqref{eq: def flat convex}) and $\He W^+$ has a uniformly bounded operator norm.
\item $W^-$ is given by the sum of a quadratic function, 
and a function which admits a bounded, Lipschitz mode decomposition in the following sense. 
	There is a sequence of functions $n_k:\R^d \to \R [-1,1]$ 
and coefficients $\alpha \geq 0$, $w^-_k \geq 0$ ($k\in\N$) such that:
\begin{equation}
\label{eq: W- decomposition 0}
W^-(x,y) 
:=
\alpha \,  (x ,  y) + \sum_{k\geq 0}w^-_k n_k(x)n_k(y):=
\sum_{k\geq -d}w^-_k n_k(x)n_k(y),
\end{equation}
	where we set $w^-_{-i}=\alpha$ and $n_{-i}(x)=x^{(i)}$ for $i\in\{1,...,d\}$, 
	and:
	\begin{equation}
\sup_{x,y\in\R^d}\big|W^- (x,y)-\alpha\, (x,y)\big|<\infty,
\qquad 
\sum_{k\geq 0}w^-_k 
\sup_{x\in\R^d}|\nabla n_k(x)|^2 < \infty
.
\label{eq: upper bound W-}
	\end{equation}
\end{itemize}
\end{assumption}
\renewcommand{\cM}{{\bf m}}
The functions $W^\pm$ do not play the same role:  $W^+$ cannot induce a phase transition contrary to the  interaction $W^-$ which may do so depending on $\alpha, w^-_k,n_k$. 
In order to determine a threshold for the validity of the log-Sobolev inequality, we are going to define the restriction of the mean-field free energy  $\cF_T$ \eqref{eq_def_free_energy}  to the modes $n_k$.  
Given $\cM = (m_k)_{k \geq -d}$, we consider the subset of probability densities with prescribed modes
\begin{equation}
  \label{eq: modes}
  {\bf P} (\cM) = \Big \{ \rho  \in \mathbf{M}_1(\R^d); \quad m_{k} =  \int_{\R^d} n_k(x) \rho(dx)  \Big \}
 \end{equation}
and define the \emph{coarse grained free energy} as 
\begin{equation}
\hat  \cF_T(\cM) 
= \inf \Big\{ \cF_T(\rho), \qquad \rho \in {\bf P} (\cM) \Big\}, 
\label{eq: projection F_T}
\end{equation}
with the convention $\cF_T (\cM) = + \infty$
if ${\bf P} (\cM) = \emptyset$.
This is a multi-mode  generalisation of \eqref{eq: coarse grained m}.

The functional $\hat  \cF_T$ is \emph{strongly convex} if there is $\delta >0$ such that for any $\cM^1  = (m^1_k)_{k \ge -d}$, 
$\cM^2 = (m^2_k)_{k \ge -d}$ and $t\in[0,1]$ then
\begin{equation}
t \hat  \cF_T(\cM^1) +  (1-t) \hat  \cF_T(\cM^2)
\geq  \hat  \cF_T \big( \alpha \cM^1 +  (1-\alpha) \cM^2 \big) 
+ \frac{\delta}{2} t(1-t)  \sum_{k \ge -d} w_k^- \big( m^1_k - m^2_k \big)^2
.
\label{eq: projection F_T convexe}
\end{equation}
For a smooth functional $\hat  \cF_T$, the previous condition 
 is equivalent
to assuming that the Hessian is bounded from below by   a diagonal matrix with coefficients   $(\delta \, w_k^-)_{k \geq  -d}$.   
We also say that $\hat  \cF_T$ is \emph{$\delta$-convex} if~\eqref{eq: projection F_T convexe} holds with a specific $\delta>0$.

\begin{theorem}
\label{thm: nonquadratic mean-field}
Let $V$ be a confinement potential
and $W$ an interaction respectively satisfying  Assumptions~\ref{ass_Vnonquad} and \ref{ass_Wnonquad}.  
If $\hat  \cF_T$ is  strongly convex, then  the mean-field measure
$m^N_T$ satisfies a log-Sobolev inequality with a constant independent of $N$. 
\end{theorem}

Theorem~\ref{thm: nonquadratic mean-field} relies on~\cite{wang2024LSI} to deal with the interaction term $W^+$ and uses a decomposition similar to the one introduced
in the proof of Theorem~\ref{thm: quadratic interaction} to handle the quadratic potential.

\begin{remark}
The quadratic interaction  considered in Theorem \ref{thm: quadratic interaction} falls into the class of the interaction potential \eqref{eq: W- decomposition 0} (by choosing $w_k^- =0$ for $k \geq 0$). 
In this case, 
 there are confinement potentials for which
the strong convexity of $\hat  \cF_T$ is a sharp condition as seen in Corollary~\ref{cor: double well}.
\end{remark}

The representation \eqref{eq: W- decomposition} is motivated by the Fourier decomposition.
In particular, in the periodic domain $[0,2 \pi)^d$, any smooth symmetric interaction potential of the form 
$W(x,y) = w(x-y)$ can be decomposed as \eqref{eq: mode decomposition 1}  (with  coefficient $\alpha =0$): for $x,y \in [0,2 \pi)^d$,
\begin{align} 
w(x-y) &= \sum_{k\geq 0} \hat w_k \cos( (k, x-y) )\nnb
&=  \sum_{k\geq 0} \hat w_k \cos((k,x))\cos((k,y)) + \hat w_k \sin((k,x))\sin((k,y)).
\label{eq: mode decomposition 1}
\end{align}
The function $W$ can then be split into $W^+, W^-$ according to the sign of the Fourier coefficients. 
The Lipschitz assumption \eqref{eq: upper bound W-} on the $n_k$  is implied by sufficient smoothness of $w$.

As a consequence Theorem \ref{thm: nonquadratic mean-field}
 (or rather its proof which also applies on the torus)
implies the following result in the periodic case.
\begin{corollary}
\label{cor: periodic case}
Consider the mean-field measure on the periodic domain $[0,2 \pi]^d$ with   smooth periodic potentials $V (x)$ and $W(x,y) = w(x-y)$. 
If $\hat  \cF_T$ is  strongly convex, then  the mean-field measure
$m^N_T$ satisfies a log-Sobolev inequality with a constant independent of $N$.  
\end{corollary}

\begin{remark}
In Appendix~\ref{append: XY}, we check that for the XY model, Corollary~\ref{cor: periodic case} implies the log-Sobolev inequality  all the way to the critical threshold $T_c=1/2$.
Note that this was already established in~\cite{MR3926125}. 
The mode decomposition did not appear there,
  but in this special situation, it is equivalent to the $\R^2$-valued external field that appeared instead.

Using spherical harmonics,
this can similarly be extended to rotation invariant interactions on $\bbS^{d}$, $d \geq 2$, i.e., $W(x,y) = W(x\cdot y)$.
In particular, for the mean-field $O(n)$ model, in which $x_i \in \bbS^{n-1}$,
the addition theorem for spherical harmonics \cite[Theorem~2.9]{MR2934227} implies that
\begin{equation}
  - W(x_i,x_j)
  = x_i\cdot x_j
  = \frac{|\bbS^{n-1}|}{N_{1,n}}\sum_{m} Y_1^m(x_i)\overline{Y_1^{m}(x_j)}
\end{equation}
where $(Y^m_1)_{1\leq m\leq N_{1,n}}$ is an orthonormal basis of the spherical harmonics of order $1$ in $n$ dimensions. 
%
This can be arranged into real form so that the right-hand side becomes $\sum_k n_k(x_i)n_k(x_j)$.
For $x_i\in \bbS^2$ this reduces to the trigonometric identity
\begin{align}
  x_i\cdot x_j
  &= \cos(\theta_i)\cos(\theta_j) + \cos(\varphi_i)\sin(\theta_i)\cos(\varphi_j)\sin(\theta_j)
    + \sin(\varphi_i)\sin(\theta_i)\sin(\varphi_j)\sin(\theta_j)
    \nnb
  &= n_1(x_i)n_1(x_j) + n_2(x_i)n_2(x_j)+n_3(x_i)n_3(x_j),
\end{align}
and the $n_k(x)$ are simply  the spherical coordinates of $x\in \bbS^2$.
In the same way as for the XY model, it was shown in~\cite{MR3926125} that the critical threshold $T_c=1/n$ for the $O(n)$ model  can be reached  using this decomposition, for any $n$.
\end{remark}

\subsection{Possible generalisations}

We conclude this section by mentioning a series of open problems to generalise Theorem \ref{thm: nonquadratic mean-field}. 

\begin{enumerate}
\item In the compact situation (torus or sphere) the mode decomposition into Fourier modes or spherical harmonics seems very natural.
  On the other hand, the assumption of bounded modes is less relevant on an unbounded space. 
  Can the proof be adapted to the case where $W^-$ only has bounded Hessian?
\item The strong convexity assumption on $\hat \cF_T$ applies to all modes simultanously.
  This is in the spirit of the Bakry--\'Emery criterion, but
  different from the scale decomposition in the Polchinski
  renormalisation group flow \cite{10.1214/24-PS27}, where the scales are effectively revealed one after another from the smallest to the largest scales.
  Is there a version of this renormalisation group strategy that would explore modes rather than scales in an ordered fashion?
\item 
The convexity criterion on $\hat  \cF_T$ has been introduced to provide a simple criterion in terms of  the mean-field free energy  $\cF_T$ \eqref{eq_def_free_energy}, but we do not expect this condition to be optimal in general. 

If the coarse grained free energy $\hat  \cF_T$ depends only on a finite number of modes and 
satisfies a Polyak-\L ojasiewicz inequality of the form \eqref{eq: Polyak-Lojasiewicz hat F}  then the same discussion as in Theorem \ref{thm: quadratic interaction} would 
imply the conclusion of Theorem~\ref{thm: nonquadratic mean-field},  i.e., the uniform log-Sobolev inequality.

More generally, it would be interesting to investigate if the log-Sobolev inequality for the particle system could be implied by an assumption on a uniform rate of exponential relaxation for the gradient flow associated with $\cF_T$ (in the sense of~\cite{MR2053570}). 
We refer to \cite[Conjecture 1]{MR4604897} for a  precise conjecture.
\end{enumerate}

\section{Quadratic interaction potential}
\label{sec: quadratic}

In this section, we prove Theorem~\ref{thm: quadratic interaction}  and Corollary~\ref{cor: double well} for quadratic interactions $W(x,y) = - (x,y)$. 
This example also illustrates the general strategy in the simplest instance of one mode.

We first prove a log-Sobolev inequality up to a certain convexity threshold on the temperature for  potentials $V$  
satisfying Assumption \ref{ass_Vnonquad} and $d$-dimensional spin variables $x_i \in \bbR^d$. The analysis is carried out in terms of an auxiliary functional, the renormalised potential, defined in \eqref{eq_def_V_t_quadratic_W}.
For a certain class of double well potentials $V$, we show in Section \ref{sec_crit_threshold} that this threshold coincides with the critical temperature $T_c$ of the free energy and that the log-Sobolev constant diverges like $(T-T_c)^{-1}$ as $T\downarrow T_c$.

\subsection{Renormalised potential and renormalised measure}
\label{sec_measure_decomp_quadratic}
Our starting point is the following elementary identity, valid for all $(x_1,\dots, x_N) \in (\R^d)^N$:
\begin{equation} \label{e:mftidentity}
  \exp\qBB{\frac{1}{2N T} \sum_{i,j=1}^N(x_i,x_j)}
  = 
  \text{constant}
  \int_{\R^d} \exp\qBB{-\frac{N|\varphi|^2}{2T} + \frac{1}{T}\Big(\varphi,\sum_{i=1}^Nx_i\Big)}\, d\varphi
  ,
\end{equation}
where the constant is $(N/(2\pi T))^{d/2}$ and is not relevant. 
The identity~\eqref{e:mftidentity} induces a decomposition of the mean-field measure $m^N_T$ \eqref{eq_measure_nonquad}. 
Indeed for any test function $F:(\R^d)^N\to\R$, one gets
\begin{align}
\E_{m^N_T}[F]
=
\frac{\text{constant}}{Z^N_T} \int_{(\R^d)^N}  
 \int_{\R^d}  F(x) \; \exp\qBB{-\frac{N|\varphi|^2}{2T} + \frac{1}{T}\Big(\varphi,\sum_{i=1}^Nx_i\Big)}\, d\varphi
\, \prod_{i=1}^N\alpha_V(dx_i).
 \label{eq_meas_decomp_quadratic 0}
\end{align}
This decouples the interaction between the spins, so that the mean-field measure can be rewritten, after exchanging the order of integration, as 
\begin{equation}
\E_{m^N_T}[F]
=
\E_{\nu^r_T} \big[ \E_{\mu^\varphi_T}[F] \big],
\label{eq_meas_decomp_quadratic}
\end{equation}
with the \emph{renormalised measure} $\nu^r_T$ 
and the \emph{fluctuation measure} $\mu^{\varphi}_T$ ($\varphi\in\R^d$) given by:
\begin{equation}
\nu^r_T(d\varphi) \propto e^{ - N V_T(\varphi)}\, d\varphi \in {\bf M}_1(\R^d)
,
\qquad
\mu^{\varphi}_T(dx) \propto \prod_{i=1}^N e^{\frac{1}{T} (\varphi,x_i)
    }  \alpha_V (dx_i)
\in {\bf M}_1\big((\R^d)^N\big).
\label{eq_renorm_measure_fluct_measure}
\end{equation}
The measure $\alpha_{V}(dx)\propto e^{-V(x)}\, dx$ is the one defined in \eqref{eq: def alpha V}. 
The renormalised potential $V_T$ is defined for $T> 0$ and $\varphi\in\R^d$ by:
\begin{equation}
  V_T(\varphi) =   \frac{|\varphi|^2}{2T}
  -\log\int_{\R^d} 
  e^{\frac{(x,\varphi)}{T}} \, \alpha_{V}(dx)  .
 \label{eq_def_V_t_quadratic_W}
\end{equation}
Note that the normalisation factor $Z^N_T$ in \eqref{eq_meas_decomp_quadratic 0} cancels with the normalisation factors of the probability measures $\nu^r_T, \mu^\varphi_T$.

The measure decomposition~\eqref{eq_meas_decomp_quadratic} says that the mean-field quadratic interaction can be realised as $N$ independent copies of the measure $\alpha_V$ coupled with an external field $\varphi$  distributed according to the probability measure $\nu^r_T$. 
In the next section, we use this decomposition to prove a uniform log-Sobolev inequality for $m^N_T$ provided
the measure $\exp( - N V_T)$ satisfies a suitable log-Sobolev inequality.

\subsection{Log-Sobolev inequality in the high temperature phase}
\label{sec_LSI_HT_quadratic}
 We are going to show that the mean-field measure $m^N_T$ satisfies a log-Sobolev inequality with constant bounded uniformly in $N$ for any temperature at which the 
renormalised measure $\nu^r_T$ satisfies a log-Sobolev inequality with constant $N \lambda_T$ for some constant $\lambda_T$ independent of $N$.
Throughout the section, $\gamma_V>0$ is such that $\alpha^h_V(dx)\propto e^{(h,x)-V(x)}\, dx$ satisfies a log-Sobolev inequality with constant $\gamma_V$ uniform in $h\in\R^d$. 
Such a $\gamma_V$ exists if the  confinement potential $V$ satisfies Assumption~\ref{ass_Vnonquad}.

\begin{proposition}
\label{prop_LSI_quadratic_W}
Let $T>0$ be such that
 $\nu^r_T(d\varphi)\propto e^{-NV_T(\varphi)}\, d\varphi$ satisfies a log-Sobolev inequality with constant $N\lambda_T$
for some $\lambda_T>0$. 
Then $m^N_T$ satisfies a log-Sobolev inequality with constant:
\begin{equation}
\frac{1}{\gamma^{N}_{\rm LS}(T)}
\leq 
\frac{1}{\gamma_V} + \frac{ 1}{\gamma_V^2T^2\, \lambda_T}
.
\end{equation}
\end{proposition}

The assumption of the proposition holds if $\He V_T\geq \lambda_T\id$ with $\lambda_T>0$
 by the Bakry--\'Emery criterion~\cite{MR889476}.
More generally, we will show in Proposition~\ref{prop: Polyak-Lojasiewicz} in the next subsection 
that it is implied by a Polyak-\L ojasiewicz inequality \eqref{eq: Polyak-Lojasiewicz hat F} for the coarse grained free energy $\hat\cF_T$.

\begin{proof}
  The measure $m^N_T$ has been split into two measures which are well behaved in the sense that they both satisfy a log-Sobolev inequality under the assumptions of Proposition \ref{prop_LSI_quadratic_W}.
  By assumption $\nu^r_T(d\varphi)$ satisfies a log-Sobolev inequality with constant $N\lambda_T>0$. 
  By Assumption~\ref{ass_Vnonquad}, the measure $\mu^{\varphi,i}_T(dx_i) \propto e^{\frac{1}{T} (\varphi,x_i)}\alpha_V(dx_i)$ satisfies a log-Sobolev inequality with constant $\gamma_V$ independent of $\varphi\in\R^d$ ($1\leq i \leq N$).
Thus the same is true for the product measure $\mu^{\varphi}_T$.

Let $G(\varphi) = \E_{\mu^{\varphi}_T}[F^2]^{1/2}$. Then the measure decomposition~\eqref{eq_meas_decomp_quadratic} implies the following standard entropy decomposition:
\begin{align}
\ent_{m^N_T}(F^2)
 = \E_{\nu^r_T}\big[\ent_{\mu^{\varphi}_T}(F^2)\big] + \ent_{\nu^r_T}(G(\varphi)^2).
\end{align}
As the measures $\nu^r_T, \mu^{\varphi}_T$ satisfy a log-Sobolev inequality, we deduce that 
\begin{align}
\ent_{m^N_T}(F^2)
  \leq
  \frac{2}{\gamma_V} \sum_{i=1}^N \E_{\nu^r_T}\E_{\mu^\varphi_T}\big[|\nabla_{x_i} F|^2\big]
  + \frac{2}{N\lambda_T} \E_{\nu^r_T}\big[|\nabla_{\varphi} G(\varphi) |^2\big]
  .
  \label{eq_first_bound_ent_quad}
\end{align}
By \eqref{eq_meas_decomp_quadratic}, the first term is precisely $(2/\gamma_V)\E_{m^N_T}[|\nabla F|^2]$. 
For the second term, 
notice:
\begin{equation}\label{e:nablaG}
  \nabla_{\varphi} G(\varphi)
  = \frac{\nabla_{\varphi} \E_{\mu^{\varphi}_T}[F^2]}{2\E_{\mu^{\varphi}_T}[F^2]^{1/2}}
  = \frac{1}{2T} \frac{\cov_{\mu^{\varphi}_T}\big(F^2,\sum_{i=1}^N x_i\big)}{\E_{\mu^{\varphi}_T}[F^2]^{1/2}}
  .
\end{equation}
The covariance is estimated by Lemma~\ref{lemm_BH} applied to $H(x)=\sum_i x_i$ which satisfies $|\nabla \sum_i x_i(a)|=\sqrt{N}$ for each $1\leq a\leq d$:
\begin{equation}
\cov_{\mu^{\varphi}_T} \Big(F^2,\sum_{i=1}^N x_i \Big)^2
\leq 
\frac{4N}{\gamma_V^2} \,  \E_{\mu^{\varphi}_T} \big[  F^2 \big]  \, 
\sum_{i=1}^N \E_{\mu^{\varphi}_T}\big[ \, \big|\nabla_{x_i} F\big|^2\, \big]
.
\end{equation}
Together with~\eqref{eq_first_bound_ent_quad} this completes the proof. 
\end{proof} 

\subsection{Renormalised potential and coarse grained free energy  - Proof of Theorem~\ref{thm: quadratic interaction}}

To prove Theorem~\ref{thm: quadratic interaction},
it suffices to show that
the assumption of Proposition~\ref{prop_LSI_quadratic_W} is implied by the
Polyak-\L ojasiewicz inequality \eqref{eq: Polyak-Lojasiewicz hat F} for $\hat\cF_T$. 
This is done in Proposition~\ref{prop: Polyak-Lojasiewicz} below.
We start with a general correspondence between the renormalised potential and the free energy that will be used extensively in
Section~\ref{sec; proof theorem 1.7} in a more general context.

\begin{lemma}
\label{lemma_link_V_T_cF_T}
Let $\varphi\in \R^d$. 
Then the renormalised potential introduced in \eqref{eq_def_V_t_quadratic_W} can be rewritten as 
\begin{align} \label{eq: projection F_T transform quadratic}
  V_T(\varphi) 
  &= \inf_{\rho\in {\bf M}_1( \R^d )} 
    \left\{ \cF_T(\rho) +\frac{1}{2T}\Big(\varphi -\int x\,\rho(dx) \Big)^2 \right \}
    \nnb
  &
  = \inf_{m\in\R^d} 
\left\{ \hat \cF_T(m) +\frac{1}{2T} | \varphi -m  |^2 \right \} 
,
\end{align}
and 
there are as many global minimisers for the free energy $\cF_T$ as for the renormalised potential $V_T$. 
\end{lemma}

The formula \eqref{eq: projection F_T transform quadratic} is reminiscent of the   Hopf--Lax formula for Hamilon--Jacobi equations,
but $\hat\cF_T$ in the argument also depends on $T$.
We refer to  \cite[Appendix A]{10.1214/24-PS27} for a discussion on the renormalisation group flow and the Hamilton--Jacobi equation.

\begin{proof}[Proof of Lemma~\ref{lemma_link_V_T_cF_T}]
Recall the definition \eqref{eq_def_V_t_quadratic_W} of $V_T$: 
for each $N\geq 1$,
\begin{equation}
V_T(\varphi) 
= 
 \frac{|\varphi|^2}{2T}
  -\log\int_{\R} 
 e^{\frac{(x,\varphi)}{T}} \, \alpha_{V}(dx_1) 
=
\frac{|\varphi|^2}{2T}
  -\frac{1}{N}\log\int_{\R^N} 
\prod_{i=1}^N  e^{\frac{(x,\varphi)}{T}} \, \alpha_{V}(dx_i) 
.
\end{equation}
Taking the large $N$ limit, 
Sanov's theorem gives:
\begin{align}
 V_T(\varphi)  
&= 
\frac{|\varphi|^2}{2T} - \sup_{\rho\in{\bf M}_1(\R^d)} 
\left\{ \frac{1}{T}\Big(\varphi, \int x \,\rho(dx)\Big) - \int V(x)\, \rho(dx) - \int \rho(x) \log \rho(x)\, dx \right \} 
\nnb
& = \inf_{\rho\in{\bf M}_1(\R^d)} 
\left\{ \cF_T(\rho)
+
\frac{1}{2T}\Big|\varphi -\int x\,\rho(dx)\Big|^2 \right \}.
\label{eq: 1d variational principle for VT}
\end{align}
This formula is the counterpart of \eqref{eq: projection F_T transform} which will be established later on in a more general framework. 
The argument of the variational principle in the first line above is strictly convex in $\rho$. 
There is thus a unique critical point $\rho_{m_\varphi}$, 
parametrised by its magnetisation $m_\varphi=\int x\, \rho_{m_\varphi}(dx)$, and explicitly given by:
\begin{equation}
\rho_{m_\varphi}(dx)
\propto
e^{\frac{(x,\varphi)}{T} - V(x)}\, dx
.
\label{eq: minimum rho_m}
\end{equation}
In terms of the coarse grained free energy \eqref{eq: projection F_T},
the variational formula \eqref{eq: 1d variational principle for VT} can be rewritten as
\begin{align}
\label{eq: identite VT F}
  V_T(\varphi) 
  & = \inf_{m} 
\left\{ \hat \cF_T(m) +\frac{1}{2T}\Big|\varphi -m \Big|^2 \right \} 
\nnb
& = \inf_{m^*} \hat \cF_T(m^*) + 
 \inf_{m} \left\{ \big( \hat \cF_T(m) - \inf_{m^*} \hat \cF_T(m^*)\big) +\frac{1}{2T}\Big|\varphi -m \Big|^2 \right \}
 .
\end{align}
This implies  that the global minima of $V_T$ coincide exactly with the  global minima of $\hat \cF_T$.
As all $\rho_m$ have different mean, 
there are therefore as many global minimisers for the free energy $\cF_T$ as for the renormalised potential $V_T$.  
\end{proof}

\begin{proposition}
\label{prop: Polyak-Lojasiewicz}  
Let $T>T_c$ and assume that  $\hat\cF_T$ satisfies a Polyak-\L ojasiewicz inequality \eqref{eq: Polyak-Lojasiewicz hat F}.
Then the  renormalised measure $\nu^r_T$ satisfies a log-Sobolev with constant $N \lambda_T$ for some constant $\lambda_T >0$ independent of $N$.
\end{proposition}

\begin{proof}  

  By definition \eqref{eq_def_beta_c}, for  $T>T_c$ the free energy $\cF_T$ has a unique minimiser.
  Therefore $\hat\cF_T$ has a unique minimiser $m^\star$, and Lemma~\ref{lemma_link_V_T_cF_T} implies that $V_T$ has a unique minimum at $m^\star$.
It is shown in \cite[Theorem 1]{CheStr24PL} that if $V_T$ satisfies a  Polyak-\L ojasiewicz inequality  for some constant $\gamma>0$,
\begin{equation}
\label{eq: Polyak-Lojasiewicz VT}
V_T(\varphi) - V_T(m^\star) \leq \frac{1}{2\gamma} \| \nabla V_T (\varphi) \|^2, \quad \forall \gp\in\R^d,
\end{equation}
 then the renormalised measure $\nu^r_T(d\varphi)\propto e^{-NV_T(\varphi)}\, d\varphi$ satisfies a log-Sobolev inequality with constant $N\gamma(1+o_N(1))$, 
which is at least $N\lambda_T$
for some $\lambda_T>0$ and all large enough $N$. 
Thus to conclude Proposition \ref{prop: Polyak-Lojasiewicz}, it is enough to show that 
\eqref{eq: Polyak-Lojasiewicz VT} holds thanks to the assumption on $\hat \cF_T$.

To see this, we use the variational formula derived in Lemma \ref{lemma_link_V_T_cF_T} above: 
for each $\varphi\in\R^d$,
\begin{align}
  V_T(\varphi) 
   = \inf_{m\in\R^d} 
\left\{ \hat \cF_T(m) +\frac{1}{2T} | \varphi -m  |^2 \right \} 
= \hat \cF_T(m_\varphi) +\frac{1}{2T} |\varphi - m_\varphi  |^2.
\label{eq: identite VT F bis}
\end{align}
If $\hat\cF_T$ is regular enough, the argmin is determined as a solution of:   
\begin{align}
\frac{1}{T} ( \varphi - m_\varphi) = \nabla \hat \cF_T(m_\varphi)
.
\end{align}
 To establish regularity of $\hat\cF_T$,  
note from its explicit expression~\eqref{eq_def_V_t_quadratic_W} that $V_T\in C^\infty(\R^d)$. 
The same is therefore true of $\varphi\mapsto m_\varphi = \int x\, \rho_\varphi(dx) = \varphi-T\nabla V_T(\varphi)$. 
Since $m_\varphi$ has differential $\nabla m_\varphi = \cov_{\rho_\varphi}$ which is positive definite for each $\varphi\in\R^d$ by Assumption~\ref{ass_Vnonquad} on $V$, 
the local inversion theorem implies that $m_\varphi^{-1}: m\mapsto \varphi_m$, where $\int x\, \rho_{\varphi_m}(dx)=m$, is also smooth. 
Equation~\eqref{eq: identite VT F bis} then implies  $\hat \cF_T\in C^1(\R^d)$ as desired. 
Assuming sufficient regularity on the potentials, one further has 
\begin{align*}
 \nabla V_T (\varphi) = \frac{1}{T} (  \varphi - m_\varphi ) 
 = \nabla \hat \cF_T(m_\varphi).
\end{align*}
If  $\hat\cF_T$ satisfies a Polyak-\L ojasiewicz inequality \eqref{eq: Polyak-Lojasiewicz hat F} with constant $\gamma_{\rm PL}$,
then we get from \eqref{eq: identite VT F bis} that
$V_T$ satisfies also a Polyak-\L ojasiewicz inequality :
\begin{align}
  V_T(\varphi) 
  -V_T(m^\star)
  \leq  \frac{1}{2\gamma_{\rm PL}} \| \nabla \hat \cF_T(m_\varphi) \|^2 + \frac{T}{2} \|\nabla V_T (\varphi)\|^2
  = \left( \frac{1}{2\gamma_{\rm PL}}  + \frac{T}{2} \right)  \|\nabla V_T (\varphi)\|^2 ,
\end{align}
where we used that $V_T(m^\star) = \hat \cF_T (m^\star)$ by Lemma \ref{lemma_link_V_T_cF_T}.
This implies the inequality \eqref{eq: Polyak-Lojasiewicz VT} and therefore completes the proof.
\end{proof}

\subsection{Critical point for double well potentials  - Proof of Corollary~\ref{cor: double well}}
\label{sec_crit_threshold}

Proposition~\ref{prop_LSI_quadratic_W} implies a log-Sobolev inequality for $m^N_T$ for each temperature such that the renormalised potential $V_T$ is uniformly convex
(by the Bakry--\'Emery criterion for $V_T$).
The aim of this section is to exhibit a class of double-well potentials $V$ for which this criterion is sharp, 
in the sense that uniform convexity of $V_T$ holds for any $T>T_c$, 
with $T_c$ the critical temperature~\eqref{eq_def_beta_c} above which the free energy~\eqref{eq_def_free_energy} has a unique minimiser.

\begin{assumption}[GHS double-well potentials]
\label{ass_doublewell}
Let $d=1$, and
in addition to Assumption~\ref{ass_Vnonquad}, 
assume the potential $V$ is in the Griffiths--Hurst--Simon (GHS) class~\cite{MR395659}. 
That is, $V\in C^1(\R,\R)$ is even with $\lim_{|x|\to \infty}V(x)=+\infty$, 
and the restriction of $V'$ to $[0,\infty)$ is convex. 
\end{assumption}

Note that the potential $V(x)= \frac{x^4}{4} - \lambda\frac{x^2}{2}$ in \eqref{eq: potentiel quadratique} satisfies Assumption~\ref{ass_doublewell}. 
A consequence of Assumption~\ref{ass_doublewell} is the following useful bound on the variance of 
$\alpha_V^h\propto e^{hx}\alpha_V(dx)$:
\begin{equation}
\forall h\in\R,\qquad 
\var_{\alpha_V^h}(x)
\leq 
\var_{\alpha_V}(x)
.
\label{eq: minimum variance h 0}
\end{equation}
\begin{proposition}\label{prop_T_c_convexity_GHS}
Suppose Assumption~\ref{ass_doublewell} applies. 
Then $T_c = \var_{\alpha_V}(x)$ and 
\begin{equation}
\label{eq: lambda T}
 \inf_{\varphi}\partial^2_\varphi V_T(\varphi) = 
\frac{T - T_c}{T^2}.
\end{equation}
As a consequence, $V_T$ is uniformly convex  for $T>T_c$ and 
the first part of Corollary~\ref{cor: double well} follows by applying Proposition~\ref{prop_LSI_quadratic_W}.
\end{proposition}
\begin{proof}[Proof of Proposition~\ref{prop_T_c_convexity_GHS}]
An elementary computation using the definition of the renormalised potential~\eqref{eq_def_V_t_quadratic_W} and \eqref{eq: minimum variance h 0} (from Assumption~\ref{ass_doublewell}) gives:
\begin{equation}
\partial^2_\varphi V_T(\varphi) =
\frac{1}{T} - \frac{1}{T^2} \var_{\alpha_V^{\varphi / T}} (x)
\geq 
\partial^2_\varphi V_T(0) = \frac{1}{T} - \frac{1}{T^2} \var_{\alpha_V} (x)
.
\label{eq : Borne hessian_V_t}
\end{equation}
This implies that $V_T$ is uniformly convex for any $T>\var_{\alpha_V}(x)$. 
 Furthermore, for $T < \var_{\alpha_V}(x)$, then $\partial^2_\varphi V_T (0) <0$ so that the even function  $V_T$ has at least two distinct minimisers.
 By Lemma~\ref{lemma_link_V_T_cF_T}, this implies that $T_c = \var_{\alpha_V}(x)$ as $\cF_T$ and $V_T$ have the same number of global minimisers.
\end{proof}
Under Assumption~\ref{ass_doublewell} on $V$, 
we further characterise the behaviour of the log-Sobolev constant close to $T_c$. 
The following two propositions complete the proof of Corollary~\ref{cor: double well}.
\begin{proposition}[Lower bound]
\label{prop: lower bound Phi4}
Suppose Assumption~\ref{ass_doublewell} on $V$. 
Recall that $\gamma_V$ denotes a uniform bound on the log-Sobolev constant of $\alpha_V^h(dx)\propto e^{hx}\alpha_V(dx)$ ($h\in\R$). 
Then,  for each $T>T_c$, the log-Sobolev constant $\gamma^N_{\rm LS}(T)$ of the mean-field measure $m^N_T$ satisfies:
\begin{equation}
\frac{1}{\gamma^N_{\rm LS}(T)}
\leq 
\frac{1}{\gamma_V}+
\frac{1
}{(T-T_c)\gamma_V^2}
.
\end{equation}
\end{proposition}
\begin{proof}
Proposition~\ref{prop_LSI_quadratic_W} implies the following bound on the log-Sobolev constant:
\begin{equation}\label{eq_bound_LSI_constant_quadratic_W}
\frac{1}{\gamma^{N}_{\mathrm{LS}}(T)}\leq \frac{1}{\gamma_V} + \frac{ 1 }{T^2\, \gamma_V^2\, \inf_{\varphi\in\R}\partial^2_{\varphi}V_T(\varphi)}
.
\end{equation}
The lower bound \eqref{eq: lambda T} concludes the derivation of Proposition \ref{prop: lower bound Phi4}. 
\end{proof}
To get a matching upper bound, we just need to find a good test function. 
\begin{proposition}[Upper bound]\label{prop_upper_bound_LSI_quadratic}
Let $V$ satisfy Assumption~\ref{ass_doublewell}. 
There is $C>0$ such that, for all $T > T_c$, one gets for  $N$  large enough 
\begin{equation}
\gamma^N_{\mathrm{LS}}(T)
\leq  C (T-T_c) .
\end{equation}
\end{proposition}
\begin{proof}
A log-Sobolev inequality implies a Poincar\'e inequality with the same constant. 
Taking as test function $F = N^{-1/2}\sum_{i}x_i$, 
we find:
\begin{equation}
\gamma^N_{\mathrm{LS}}(T)
\leq 
\frac{1}{\chi^N_T},
\qquad
\chi^N_T
:=
\sum_{i=1}^N\cov_{m^N_T}(x_1,x_i)
=
\E_{m^N_T}\Big[\Big(\frac{1}{\sqrt{N}} \sum_{i=1}^N  x_i\Big)^2\Big]
.
\label{eq_link_LSI_suscep}
\end{equation}
The susceptibility $\chi^N_T$ can be bounded using the measure decomposition~\eqref{eq_meas_decomp_quadratic} as we now explain. 
One has 
\begin{align}
\chi^N_T
&=
\E_{\nu_T}\Big[\var_{\mu^\varphi_T}\Big(\frac{1}{\sqrt{N}}\sum_i x_i\Big)\Big]
+\var_{\nu_T}\Big(\E_{\mu^\varphi_T}\Big[\frac{1}{\sqrt{N}}\sum_i x_i\Big]\Big)
\nnb
&\geq
N\var_{\nu_T}\Big(\E_{\mu^\varphi_T}[x_1] \Big)
=
N \E_{\nu_T} \Big[ \E_{\mu^\varphi_T}[x_1]^2 \Big] ,
\label{eq: borne inf variance}
\end{align}
using the symmetry between the random variables and the fact that $\E_{\nu_T}\big[\E_{\mu^\varphi_T}[x_1] \big] = \E_{m^N_T}[x_1] = 0$.

Estimating \eqref{eq: borne inf variance} from below boils down to studying one dimensional measures. 
For $\varphi$ small, one can easily estimate by Taylor expansion,  the behavior of the expectation
\begin{equation}
\E_{\mu^\varphi_T}[x_1] = \E_{\alpha_V^{\varphi/T}}[x_1] \geq \frac{1}{2 T}  \var_{\alpha_V^0}[x_1] \, \varphi.
\label{eq: variance phi petit}
\end{equation}
For $T > T_c$, the renormalised potential is a convex function reaching its minimum at 0 and the second derivative at 0 is given by $\lambda_T = \frac{T- T_c}{T^2}$
\eqref{eq: lambda T}.
Thus the field $\gp$ under the renormalised measure is concentrated close to 0 on a set of size $1/\sqrt{N \lambda_T}$. 
As a consequence, for some $c(T)>0$
\begin{equation}
\E_{\nu_T}[\varphi^2{\bf 1}_{|\varphi|\geq 1}]
\leq 
e^{-Nc(T)}
.
\end{equation}
From this upper bound and  \eqref{eq: variance phi petit}, we get  
\begin{equation}
\E_{\nu_T}\Big[\E_{\mu^\varphi_T}[x_1]^2\Big]
\geq
\frac{\var_{\alpha_V^0}[x_1]^2} {4 T^2 \int_\R e^{-NV_T(\varphi)}\, d\varphi}
\times 
\int_{\R} \varphi^2 e^{-NV_T(\varphi)}\, d\varphi 
+e^{-Nc(T)}.
\label{eq_suscep_is_ratio}
\end{equation}
The field $\gp$ in the right-hand side and
 \eqref{eq_suscep_is_ratio} can be estimated by using Laplace method (see e.g. \cite[Theorem~3, p.495]{Wong1989AsymptoticAO}) when $N$ tends to infinity :
 \begin{equation}
\E_{\nu_T}\Big[\E_{\mu^\varphi_T}[x_1]^2\Big]
\geq
\frac{\var_{\alpha_V^0}[x_1]^2} {4 T^2 \, N \lambda_T}(1+o_N(1)) = \frac{C(1+o_N(1))}{N (T - T_c)}.
\end{equation}
Combined with \eqref{eq_link_LSI_suscep} and \eqref{eq: borne inf variance}, this completes the proof of Proposition \ref{prop_upper_bound_LSI_quadratic}.
\end{proof}

\begin{remark}
By adapting the proof of Proposition \ref{prop_upper_bound_LSI_quadratic}, one could also compute the divergence with respect to $N$ of the log-Sobolev constant at $T_c$.
\end{remark}

\section{Quadratic interaction potential on random graphs}\label{sec_graphs}
In this section, we prove Theorem~\ref{thm: quadratic interaction random} 
for models with quadratic interactions indexed by graphs satisfying Assumption~\ref{ass_graph} below. 
For a confinement potential $V$ satisfying Assumption~\ref{ass_doublewell},
we first prove an explicit bound on the log-Sobolev constant when $T>T_c$ in 
Theorem~\ref{theo_graph_sec_graph}. 
The $T<T_c$ case is treated in Section~\ref{sec_ER_RRG}. 
Here $T_c$ refers to the critical temperature of the fully connected mean-field model introduced in \eqref{eq_def_beta_c}.

\subsection{Assumption on graphs}

For a graph $G_N$, denote its adjacency matrix by $A$:
\begin{equation}
A_{ij}
=
{\bf 1}_{i\sim j}
,\quad 
A_{ii}
=
0
,\qquad 
i,j\in G_N
.
\end{equation}
We consider sequences $(G_N)_{N\geq 1}$ under the sole assumption that the largest eigenvalue in the spectrum of $A$ is isolated in the following sense. 
Let $\|M\|$ denote the operator norm of a matrix $M$:
\begin{equation}
\|M\|
=
\sup_{\substack{x,y\in\R^N \\ |x|=1=|y|}} (Mx,y)
.
\end{equation}
Let also $P = \frac{1}{N}{\bf 1}\otimes{\bf 1}$ denote the orthonormal projector on the constant mode ${\bf 1} = (1,\dots,1)$. 
\begin{assumption}\label{ass_graph}
There are sequences $d_N,\epsilon_N>0$ ($N\geq 1$) such that $\lim_{N\to\infty}\epsilon_N=0$ and the adjacency matrix of the graph satisfies  
\begin{equation}
\|A - d_N P\|
\leq 
\epsilon_N d_N
.
\label{eq_bound_spectrum_in_ass_graph}
\end{equation}
\end{assumption}
Assumption~\ref{ass_graph}  holds with large probability for different types of random graphs as stated next. 
\begin{lemma}\label{lemm_sp_ER_RRG}
Let $\E_N$ denote the expectation associated with a random regular graph with degree $d_N$ or an Erd\"os-R\'enyi random graph with mean degree $d_N$ and assume $\lim_{N\to\infty}d_N/\log N=+\infty$. 
Then:
\begin{equation}
\E_N\Big[\, \|A - d_N P\|\, \Big]
=
O(\sqrt{d_N})
.
\end{equation}
In the random regular graph case assuming only $\lim_N d_N=+\infty$ is in fact sufficient. 
\end{lemma}

\begin{proof}
For the random regular graph,
the claim is proven in~\cite[Theorem A]{MR3909972} (if $d_N\geq N^{\alpha}$ for some $\alpha>0$) and~\cite[Theorem 1.1]{Cook_gap_RRG} (if $1\ll d_N =O(N^{2/3})$) which state: 
\begin{equation}
\E_N\big[\, \| A-\E_N[A]\|\, \big] 
=
O(\sqrt{d_N})
.
\label{eq_concentration_operator_norm_A}
\end{equation}
In the Erd\"os-R\'enyi case, 
the claim is a special case of Theorem 3.2 in~\cite{spectrum_adj_ER} which in particular states:
\begin{equation}
\E_N\big[ \| A-\E_N[A] \, \|\, \big]
=
2 \sqrt{d_N}\, (1+o_N(1))
.
\label{eq_concentration_Erdos_Renyi}
\end{equation}
\end{proof}
\begin{remark}\label{rmk_more_graphs}
Assumption~\ref{ass_graph} says that the eigenvalues of $A -d_NP$ are negligible with respect to $d_N$. 
This is a natural condition for us because $P=\frac{1}{N}{\bf 1}\otimes{\bf 1}$ corresponds to $\E_N[A]$ when $A$ is the adjacency matrix of the models of random graphs   considered in Theorem~\ref{thm: quadratic interaction random}. 

However, the specific choice of the projector $P$ onto the span of $\bf 1$ is not necessary.
For example, if the  graph is a good expander, in the sense that: 
\begin{equation}
\|A-d_NP_N\|
\leq 
\epsilon_N d_N
,
\end{equation}
where $d_N$ is now the Perron--Frobenius eigenvalue of $A$ and $P_N$ is now the orthonormal projector on the corresponding eigenspace (which is not necessarily $\bf 1$ if the graph is not regular),
then one  can check that the uniform log-Sobolev inequality of Theorem~\ref{thm: nonquadratic mean-field}(i) for $T>T_c$ still holds under this assumption, 
with a nearly identical proof. 
\end{remark}

\subsection{Measure decomposition}
Introduce the reduced adjacency matrix:
\begin{equation}
B
:=
A- d_NP,\qquad 
P
=
\frac{1}{N}{\bf 1}\otimes{\bf 1}
.
\label{eq: def B}
\end{equation}
Using the relation $A=B+d_NP$,  we proceed as in \eqref{eq_renorm_measure_fluct_measure} and decompose the Gibbs measure introduced in \eqref{eq: measure_graph}
into two measures:
\begin{equation}
\E_{m^{G_N}_T}[F]
=
\E_{\nu^B_{r,T}}\Big[\E_{\mu^{B,\varphi}_T}[F]\Big],
\label{eq: graph mesure decomposition}
\end{equation}
where the fluctuation measure is defined in terms of the external field $\varphi \in \R$
as 
\begin{equation}
\mu^{B,\varphi}_T(dx)
\propto
\exp \bigg[  \frac{1}{2Td_N} 
(x, B x) +\frac{\varphi (x,{\bf 1})}{T}\bigg] 
\prod_{i=1}^N \alpha_V(dx_i) ,
\label{eq_def_fluct_measure_graph}
\end{equation}
and the renormalised measure $\nu^B_{r,T}\propto e^{-NV^B_T(\varphi)}\, d\varphi$ now involves an $N$-dependent renormalised potential $V^B_{T}$ that reads:
\begin{equation}
e^{-NV^B_{T}(\varphi)}
=   \exp \left[ -\frac{N \, \varphi^2}{2T} \right] \;
\int \exp \left[ \frac{1}{2Td_N} 
(x,B x) +\frac{\varphi (x,{\bf 1})}{T}\right] 
\prod_{i=1}^N \alpha_V(dx_i)
.
\end{equation}

Compared with \eqref{eq_renorm_measure_fluct_measure}, the fluctuating measure is no longer product.  Nevertheless, 
we will see that, under Assumption \ref{ass_graph}, 
the contribution of  $B$ can be neglected and the behaviour of the renormalised potential is accurately described in terms of the product measure:
\begin{equation}
\mu^{0,\varphi}_T(dx) \propto
\exp\Big[  \frac{\varphi(x,{\bf 1})}{T}\Big] \prod_{i=1}^N\alpha_V(dx_i)
. 
\label{eq_def_fluct_measure_P}
\end{equation}
We write a superscript $0$ to emphasise the difference with $\mu^{B,\varphi}_T$, 
but note that this measure is exactly the fluctuation measure appearing in the decomposition~\eqref{eq_renorm_measure_fluct_measure} of the mean-field measure $m^N_T$~\eqref{eq_meas_decomp_quadratic}  (i.e.~on the complete graph). 
We similarly write $\nu^0_{r,T}$ for the mean-field renormalised measure: 
\begin{equation}
\nu^0_{r,T}(d\varphi)
\propto 
e^{-NV_T^0 (\varphi)}\, d\varphi
,
\end{equation}
with the mean-field renormalised potential $V^0_T$ given by~\eqref{eq_def_V_t_quadratic_W}:
\begin{equation}
e^{-NV^0_T(\varphi)}
=
\int_{\R^N} \exp\Big[-\frac{N \, \varphi^2}{2T} + \frac{\varphi(x,{\bf 1})}{T}\Big]\,\prod_{i=1}^N \alpha_V(dx_i)
,\qquad 
\varphi\in\R
.
\label{eq_def_VP_T}
\end{equation}

The next theorem shows that the log-Sobolev inequality for $m^{G_N}_T$ is determined by the critical temperature  of the mean-field model.
\begin{theorem}
\label{theo_graph_sec_graph}
Let $G_N$ satisfy Assumption~\ref{ass_graph} and $V$ satisfy Assumption~\ref{ass_doublewell}. 
Let $T>T_c$, 
the critical temperature \eqref{eq_def_beta_c} in the mean-field case. 
Then, for $N$ large enough depending only on $T,\gamma_V$, the measure $m^{G_N}_T$ satisfies a log-Sobolev inequality with constant:
\begin{equation}
\frac{1}{\gamma^{G_N}_{\rm LS}(T)}
\leq 
\frac{1}{(T-T_c)\gamma_V^2} + \frac{1}{\gamma_V} + O(\epsilon_N)
,
\label{eq_bound_LSI_graph_sec_graph}
\end{equation}
where the constant $\gamma_V$ depends only on the confinement potential $V$.
\end{theorem}
By Lemma \ref{lemm_sp_ER_RRG}, Assumption~\ref{ass_graph} is satisfied for  random regular graphs and  Erd\"os-R\'enyi random graphs with large degrees. Thus Theorem~\ref{theo_graph_sec_graph} implies part (i) of Theorem~\ref{thm: quadratic interaction random}. 
Part (ii) is derived in Section \ref{sec_ER_RRG}.

\subsection{Proof of Theorem \ref{theo_graph_sec_graph}}
Using the representation \eqref{eq: graph mesure decomposition},
the entropy under $m^{G_N}_T$ decomposes as:
\begin{equation}
\ent_{m^{G_N}_T}(F^2)
=
\E_{\nu^B_{r,T}}\big[\ent_{\mu^{B,\varphi}_T}(F^2)\big] + \ent_{\nu^B_{r,T}}\big(\E_{\mu^{B,\varphi}_T}[F^2]\big)
.
\label{eq_entropy_decom_graph}
\end{equation}
The following two propositions, proven in the next sections, 
provide an estimate for each of the above terms.

The next proposition is basically~\cite{MR3926125}. 
The result there is stated on a compact state space and established using slightly different properties of Gaussian measures,  
so we reprove it at the end of the section. 
\begin{proposition}\label{prop_reduced_adj_LSI}
The measure $\mu^{B,\varphi}_T$ satisfies a log-Sobolev inequality with constant $\gamma^B_{\rm LS}(T)$ independent of $\varphi$ and bounded as follows, 
as soon as the expression between parentheses is strictly positive:
\begin{equation}
\frac{1}{\gamma^B_{\rm LS}(T)}
\leq 
\frac{1}{\gamma_V^2}\Big(\frac{1}{3T\epsilon_N}-\frac{1}{T^2\gamma_V}\Big)^{-1} + \frac{1}{\gamma_V}
=
\frac{1}{\gamma_V} + O(\epsilon_N),
\end{equation}
where the constant $\gamma_V$ depends only on the confinement potential $V$.
\end{proposition}
The next proposition controls the renormalised measure and is proven in Section~\ref{sec_renom_measure_graph}.
\begin{proposition}\label{prop_comparison_renorm_pot_graph}
Let $T>0$ and suppose that $V$ satisfies Assumption \eqref{ass_doublewell}. 
If $V^0_T$ is uniformly convex, then for $N$ large enough so is $V^B_T$. 
Explicitly, 
for any $N\geq 1$:
\begin{equation}
\inf_{\varphi\in\R}\partial^2_\varphi V^B_T(\varphi)
\geq 
\inf_{\varphi\in\R}\partial^2_\varphi V^0_T(\varphi)
- O(\epsilon_N)
.
\end{equation}
\end{proposition}
\begin{proof}[Proof of Theorem~\ref{theo_graph_sec_graph}]
Fix $T>T_c$. Since $V_T^0$ coincides with the renormalised mean-field potential \eqref{eq_def_V_t_quadratic_W},
we get by  \eqref{eq: lambda T} that:
\begin{equation}
\label{eq: borne sur V0}
\inf_{\varphi\in\R}\partial^2_\varphi V_T^0 (\varphi)
= \frac{T - T_c}{T^2} > 0 .
\end{equation}
Proposition~\ref{prop_reduced_adj_LSI} controls the first term in the right-hand side of~\eqref{eq_entropy_decom_graph}:
\begin{equation}
\E_{\nu^B_{r,T}}\big[\ent_{\mu^{B,\varphi}_T}(F^2)\big]
\leq 
\frac{2}{\gamma^B_{\rm LS}(T)}\E_{m^{G_N}_T}\big[|\nabla F|^2\big]
=
\frac{2}{\gamma_V+O(\epsilon_N)}\E_{m^{G_N}_T}\big[|\nabla F|^2\big]
.
\end{equation}
On the other hand, 
Proposition~\ref{prop_comparison_renorm_pot_graph} enables us to apply the Bakry--\'Emery criterion to the renormalised measure: 
\begin{align}
\ent_{\nu^B_{r,T}}\big(\E_{m^{B,\varphi}_T}[F^2]\big)
&\leq 
\frac{2}{N\inf_{\varphi} \partial^2_\varphi V^B_T(\varphi)}
\E_{\nu^B_{r,T}}\Big[ \, \big|\nabla_\varphi\sqrt{\E_{\mu^{B,\varphi}_T}[F^2]}\big|^2\, \Big]
\nnb
&\leq 
\frac{2}{N\big(\inf_{\varphi} \partial^2_\varphi V^0_T(\varphi) + O(\epsilon_N)\big)}
\E_{\nu^B_{r,T}}\Big[ \, \big| \nabla_\varphi\sqrt{\E_{\mu^{B,\varphi}_T}[F^2] }\big|^2\, \Big]
.
\label{eq_entropy_decom_graph_bis}
\end{align}
We conclude on a log-Sobolev inequality for $m^{G_N}_T$ using Lemma~\ref{lemm_BH} exactly as in the quadratic case,  see Section~\ref{sec_LSI_HT_quadratic}:
\begin{equation}
\big| \,
\nabla_\varphi\sqrt{\E_{\mu^{B,\varphi}_T}[F^2]}  \, \big|^2
  = \frac{1}{4T^2} \frac{\cov_{\mu^{B,\varphi}_T}\big(F^2,\sum_{i=1}^N x_i\big)^2}{\E_{\mu^{B,\varphi}_T}[F^2]}
  \leq \frac{N}{\gamma^{B}_{\rm LS}(T)^2 \,  T^2} \, 
  \E_{\mu^{\varphi}_T}\big[ \, \big| \nabla  F\big|^2\, \big] .
\end{equation}
Thus \eqref{eq_entropy_decom_graph_bis} becomes
\begin{align}
\ent_{\nu^B_{r,T}}\big(\E_{m^{B,\varphi}_T}[F^2]\big)
\leq \frac{2}{\gamma^{B}_{\rm LS}(T)^2 \,  T^2} \, 
\frac{1}{\big(\inf_{\varphi} \partial^2_\varphi V^0_T(\varphi) + O(\epsilon_N)\big)}
 \E_{m^{G_N}_T}\Big[ \, \big| \nabla F\big|^2 \, \Big]
.
\end{align}
By \eqref{eq: borne sur V0},
this proves Theorem~\ref{theo_graph_sec_graph} with constant:
\begin{equation}
\frac{1}{\gamma^{G_N}_{\rm LS}(T)}
\leq 
\frac{1}{ \gamma_V^2 \, (T-T_c) + O(\epsilon_N) } + \frac{1}{\gamma_V+O(\epsilon_N)}
.
\label{eq_explicit_bound_LSI_graph_sec_graph}
\end{equation}
\end{proof}
\subsection{Estimates for the fluctuation measure}
\label{sec_fluct_measure_graph}
\begin{proof}[Proof of Proposition~\ref{prop_reduced_adj_LSI}]
To prove the log-Sobolev inequality for  $\mu^{B,\varphi}_T$, we proceed as before and split the measure into two parts in order to decouple the scales. 
The dominant contribution is given by 
 a product measure coupled to a fluctuating weak external field. 
Recall that from Assumption \ref{ass_graph}, this matrix has spectrum in $[-\epsilon_N d_N,\epsilon_Nd_N]$. 
Define the following shifted matrix:
\begin{equation}
C
:=
\frac{1}{d_N}B+2\epsilon_N \id
\geq 
\epsilon_N \id
\geq 
0.
\end{equation}
Introduce also the potential $U$ as:
\begin{equation}
U(y)
=
V(y)+ \frac{\epsilon_N}{T}y^2,\qquad 
y\in\R
.
\label{eq_def_U_graph}
\end{equation}
The measure $\mu^{B,\varphi}_T$ then reads:
\begin{equation}
\mu^{B,\varphi}_T(dx)
\propto
\exp \left[ \frac{1}{2T} 
(x,C x) +\frac{\varphi(x,{\bf 1})}{T}\right] 
\prod_{i=1}^N \alpha_{U}(dx_i)
.
\end{equation}
We will prove a slighty stronger statement than Proposition~\ref{prop_reduced_adj_LSI} and establish the log-Sobolev inequality for the probability measure:
\begin{equation}
\mu^{C,h}_T(dx)
\propto
\exp \left[\frac{1}{2T} (x,C x) +\frac{1}{T}(h,x)\right] 
\prod_{i=1}^N \alpha_{U}(dx_i),
\end{equation}
where the field $h$ now takes values in $\R^N$. 
For $h = \varphi {\bf 1}$ then $\mu^{B,\varphi}_T = \mu^{C,h}_T$.

By Assumption~\ref{ass_graph}, the matrix $C$ has spectrum in $[\epsilon_N,3\epsilon_N]$. 
The following moment generating function formula for Gaussian random variables holds:
\begin{equation}
\exp\Big[\frac{1}{2T}(x,Cx)\Big]
\propto
\int_{\R^N}
\exp\Big[ -\frac{(y,C^{-1}y)}{2T} +\frac{1}{T}(x,y)\Big]\, dy
,\qquad 
x\in\R^N .
\end{equation}
We use it to decompose $\mu^{C,h}_T$ as follows:
for any test function $F:\R^N\to\R$,
\begin{equation}
\E_{\mu^{C,h}_T}[F]
=
\E_{\nu^{C}_{r,T}}\big[\E_{\mu^{h+y}_T}[F]\big]
.
\end{equation}
Above, the fluctuation measure $\mu^{h+y}_T$ is product:
\begin{equation}
\mu^{h+y}_T(dx)
\propto
\exp \Big[\frac{1}{T}(x,y+h) \Big] \prod_{i=1}^N \alpha_{U}(dx_i)
.
\label{eq_fluct_measure_graph}
\end{equation}
The renormalised measure $\nu^{C}_{r,T}$ is this time a probability measure on $\R^{N}$: 
\begin{align}
\nu^{C}_{r,T}(dy)
\propto
\exp \big[ - V^N_T(y) \big]\, dy
.
\end{align}
The renormalised potential $V^N_T$ is also a function from $\R^N$ to $\R$, 
given by:
\begin{equation}
e^{-V^N_T(y)}
:=
\exp\Big[-\frac{(y,C^{-1}y)}{2T}\Big]\int_{\R^N}\exp\Big[
\frac{1}{T}(x,y+h) \Big]\prod_{i=1}^N \alpha_{U}(dx_i)
.
\end{equation}
Since $C^{-1}$ satisfies:
\begin{equation}
C
\leq 
3\epsilon_N \id 
\quad\Rightarrow\quad
C^{-1}
\geq 
\frac{1}{3\epsilon_N} \id,
\label{eq: lower bound C'}
\end{equation}
 the Hessian of $V^N_T$ reads
\begin{equation}
\He V^N_T(y)
=
\frac{1}{T}C^{-1}-\frac{1}{T^2}\cov_{\mu_T^{y+h}}
\geq 
\frac{1}{3T\epsilon_N} \id -\frac{1}{\gamma_V T^2}  \id ,
\label{eq: Hessien VNT random}
\end{equation}
where $\cov_{\mu_T^{y+h}} = \big( \cov_{\mu_T^{y+h}}(x_i,x_j) \big)_{i,j \leq N}$ stands for the (diagonal) covariance matrix of the product measure $\mu^{h+y}_T$ (recall~\eqref{eq_fluct_measure_graph}), 
which can easily be bounded from above by a Poincar\'e inequality.
Note in particular that, informally speaking,~\eqref{eq: Hessien VNT random} implies that the measure $\nu^{C}_{r,T}$ is concentrated around 0, 
so that the measure $\mu^{C,h}_T=\nu^C_{r,T}\mu^{h+y}_T$ is well described by $\mu^{h+y}_T$ when $y$ is negligible. 
We make this precise next when proving a log-Sobolev inequality. 
To do so, decompose the  entropy as in \eqref{eq_entropy_decom_graph}:
\begin{equation}
\ent_{\mu^{C,h}_T}(F^2)
=
\E_{\nu^{C}_{r,T}}\big[\ent_{\mu^{h + y}_T}(F^2)\big] + \ent_{\nu^{C}_{r,T}}\big(\E_{\mu^{h +y}_T}[F^2]\big) 
.
\label{eq_entropy_decom_graph 2}
\end{equation}

First note that Assumption \ref{ass_Vnonquad} remains valid for the potential $U$ defined in \eqref{eq_def_U_graph}. 
The product measure $\mu^{h+y}_T$ therefore satisfies 
a log-Sobolev inequality uniformly in $h +y$ with constant $\gamma_V>0$: 
\begin{equation}
\E_{\nu^{C}_{r,T}}\big[\, \ent_{\mu^{h+y}_T}(F^2)\, \big]
\leq 
\frac{2}{\gamma_V}
\E_{\mu^{C,h}_T}\big[|\nabla F|^2\big]
.
\end{equation}

On the other hand, thanks to \eqref{eq: Hessien VNT random},
the renormalised measure is strictly log-concave for all $N$ large enough and the Bakry--\'Emery criterion gives:
\begin{align}
\ent_{\nu^{C'}_{r,T}}\big(\E_{\mu^{y+h}_T}[F^2]\big)
&\leq 
2 \Big(\frac{1}{3T\epsilon_N}-\frac{1}{T^2\gamma_V}\Big)^{-1}
\; \E_{\nu^{C'}_{r,T}}\Big[\, \Big|\nabla \sqrt{\E_{\mu^{y+h}_T}[F^2]}\Big|^2\, \Big]
\nnb
&\leq 
\frac{2}{\gamma_V^2}\Big(\frac{1}{3T\epsilon_N}-\frac{1}{T^2\gamma_V}\Big)^{-1}
\; \E_{\mu^{C,h}_{T}}\big[\, |\nabla F |^2\, \big]
.
\label{eq: nu C' r T LSI}
\end{align}
Above, the second line follows from Lemma~\ref{lemm_BH}, Equation~\eqref{e:F2}: 
\begin{align}
|\nabla \sqrt{\E_{\mu^{y+h}_T}[F^2]}|^2
&=
\frac{1}{4\E_{\mu^{y+h}_T}[F^2]}\sum_{i=1}^N \cov_{\mu^{y+h}_T} (F^2,x_i)^2
\leq 
\frac{1}{\gamma_V^2} \E_{\mu^{y+h}_T}\big[\, |\nabla F|^2\, \big]
.
\end{align}
We conclude that $\mu^{C,h}_T$ satisfies a log-Sobolev inequality with constant uniform in $h$:
\begin{equation}
\frac{1}{\gamma}
= 
\frac{4}{\gamma_V^2}\Big(\frac{1}{3T\epsilon_N}-\frac{1}{T^2\gamma_V}\Big)^{-1} + \frac{1}{\gamma_V}
=
\frac{1}{\gamma_V} + O(\epsilon_N)
.
\end{equation}
Since $\mu^{B,\varphi}_T = \mu^{C,h}_T$ for $h = \varphi {\bf 1}$, this completes the proof of Proposition~\ref{prop_reduced_adj_LSI}.
\end{proof}

\subsection{Comparison of renormalised potentials}\label{sec_renom_measure_graph}
\begin{proof}[Proof of Proposition~\ref{prop_comparison_renorm_pot_graph}]
The second derivative of the renormalised potential was already computed in \eqref{eq : Borne hessian_V_t}:
\begin{align}
\label{eq: main contribution variance}
\partial^2_\varphi V^0_T(\varphi)
&=
\frac{1}{T} - \frac{1}{NT^2}\var_{\mu^{0,\varphi}_T}\big((x,{\bf 1})\big)
=
\frac{1}{T} - \frac{1}{T^2}\var_{\mu^{0,\varphi}_T}\big(x_1\big)
.
\end{align}
One has also
\begin{align}
\partial^2_\varphi V^B_T(\varphi)
&=
\frac{1}{T} - \frac{1}{NT^2}\var_{\mu^{B,\varphi}_T}\big((x,{\bf 1})\big) .
\label{eq: main contribution variance B}
\end{align}
We again use the measure decomposition and the notations of Section~\ref{sec_fluct_measure_graph} to compute the variance with $h = \varphi {\bf 1}$:
\begin{align}
\label{eq: Var decomposition C}
\var_{\mu^{B,\varphi}_T}\big((x,{\bf 1})\big)
= 
\var_{\mu^{C,h}_T}\big((x,{\bf 1})\big)
& =
\E_{\nu^{C'}_{r,T}} \big[ \var_{\mu^{0,h+y}_T} \big((x,{\bf 1})\big) \big] 
+
\var_{\nu^{C'}_{r,T}} \big( \E_{\mu^{0,h+y}_T}[(x,{\bf 1})] \big) 
.
\end{align}
Each term will be estimated separately. Using the product structure of  $\mu^{0,h+y}_T$, the first term simplifies 
\begin{align}
\E_{\nu^{C'}_{r,T}}  \big[ \var_{\mu^{0,h+y}_T} \big((x,{\bf 1})\big) \big] 
 &= \sum_{i=1}^N \E_{\nu^{C'}_{r,T}} \big[ \var_{\mu^{0,h+y}_T}\big(x_i\big)\big] 
 \leq 
 N\var_{\mu^{0,0}_T}(x_1)
 ,
\end{align}
where we used the Assumption~\ref{ass_doublewell} on $V$ to bound the variance by its value at $0$ field by \eqref{eq: minimum variance h 0}. 

Let $F= (x,{\bf 1})$. 
Proceeding as in \eqref{eq: nu C' r T LSI},  the last term in \eqref{eq: Var decomposition C}
 is bounded by a spectral gap estimate
\begin{align}
\var_{\nu^{C'}_{r,T}}\big(\E_{\mu^{0,y+h}_T}[F]\big)
&\leq 
\Big(\frac{1}{3T\epsilon_N}-\frac{1}{T^2\gamma_V}\Big)^{-1} \; \E_{\nu^{C'}_{r,T}}\Big[\, \big|\nabla \E_{\mu^{0,y+h}_T}[F] \big|^2\, \Big]
\nnb
&=
 \Big(\frac{1}{3T\epsilon_N}-\frac{1}{T^2\gamma_V}\Big)^{-1} \sum_{i=1}^N\E_{\nu^{C'}_{r,T}}\big[\, \cov_{\mu^{0,y+h}_T}(F,x_i)^2 \, \big] \nnb
& \leq 
c\, \epsilon_N \sum_{i=1}^N\E_{\nu^{C'}_{r,T}}\big[\, \var_{\mu^{0,y+h}_T}(x_i) \, \big] \nnb
& \leq 
c\, \epsilon_N \, N\var_{\mu^{0,0}_T}(x_1) ,
\end{align}
where we used that the  measure $\mu^{0,y+h}_T$ is product to simplify the covariance
and then  the Assumption~\ref{ass_doublewell} on $V$ to derive an estimate on the variance uniform in $h$ by \eqref{eq: minimum variance h 0}. 
Thus the variance \eqref{eq: Var decomposition C} is bounded from above by
\begin{equation}
\sup_{\varphi \in\R} \var_{\mu^{B,\varphi}_T}[ (x,{\bf 1}) ]
\leq N \var_{\mu^{0,0}_T} (x_1)  + N\, O (\epsilon_N)
.
\end{equation}
Thanks to \eqref{eq: main contribution variance}, \eqref{eq: main contribution variance B}, this gives the claim:
\begin{equation}
\inf_{\varphi\in\R}\partial^2_\varphi V^B_T
\geq 
\inf_{\varphi\in\R}\partial^2_\varphi V^0_T -O (\epsilon_N)
.
\end{equation}
\end{proof}

\subsection{Log-Sobolev constant for $T<T_c$}
\label{sec_ER_RRG}
Let $T<T_c$ be fixed. 
In this section we prove that there are $C,\alpha>0$ depending only on the potential $V$ and on $T$ such that, 
if a sequence of graphs $G_N$ verifies Assumption~\ref{ass_graph} with sequence $\epsilon_N$,
then the Poincar\'e constant $\gamma^{G_N}_{\rm P}$ satisfies:
\begin{equation}
\gamma^{G_N}_{\rm P}(T)
\leq 
e^{C\epsilon_N N}e^{-\alpha N}
,
\label{eq_upper_bound_LSI_random}
\end{equation}
where $\gamma^{G_N}_{\rm P}(T)$ is the best constant $\gamma>0$ such that:
\begin{equation}
\var_{m^{G_N}_T}(F)
\leq 
\frac{1}{\gamma} \E_{m^{G_N}_T}\big[\, | \nabla F|^2\,\big]
,
\qquad 
F\in C^\infty_c(\R^N,\R)
.
\end{equation}
Recalling the classical bound $\gamma^{G_N}_{\rm P} \geq \gamma^{G_N}_{\rm LS}$,
~\eqref{eq_upper_bound_LSI_random} implies case (ii) in Theorem~\ref{thm: quadratic interaction random} by Lemma~\ref{lemm_sp_ER_RRG} as in particular $C\epsilon_N = O(1/\sqrt{d_N}\, )\leq \alpha/2$ with probability converging to $1$ as $N$ is large.

\begin{proof}[Proof of~\eqref{eq_upper_bound_LSI_random}]
Let $T<T_c$. 
We first build a suitable test function to show that the spectral gap (thus the log-Sobolev constant) for the (fully connected) mean-field measure $m^N_T$ is exponentially small in $N$. 
We then use a similar test function to deduce the same property 
for the measure $m^{G_N}_T$ on the random graphs.

\medskip

Let $m_\pm=m_{\pm}(T)$ denote the two values around which $\frac{1}{N}\sum_{i=1}^Nx_i$ concentrates under $m^N_T$  (note that parity of $V$ implies $m_-=-m_+$).  
Let $\delta>0$ be small enough that $|m_+-m_-|\geq 3\delta$. 
The large deviation principle for the empirical measure under $m^N_T$ with good rate function~\cite{LIU2020503} implies that there is $c_\delta>0$ such that:
\begin{equation}
 \bbP_{m^N_{T}} \Big[ \frac{1}{N}\sum_{i=1}^Nx_i \notin B(m_{\pm},\delta) \Big]
\leq 
\frac{1}{c_\delta} e^{-Nc_\delta}
,\qquad 
N\geq 1
,
\label{eq_LD_bounds_mNT}
\end{equation}
since the above event is at positive distance from minimisers of the rate function. 

Let $r<c_\delta/2$ and define:
\begin{equation}
F_r (x_1, \dots ,x_N)
= f_r \Big(\frac{1}{N}\sum_i x_i\Big)
\quad \text{with} \quad 
f_r (u) =
\begin{cases}
e^{rN}\quad &\text{ if } u \geq m_+ - \delta,\\
-e^{rN}\quad &\text{ if } u \leq m_- + \delta,\\
f_r&\text{ linear otherwise}.
\end{cases}
\end{equation}
The assumption $|m_+-m_-|\geq 3\delta$ ensures that such an $f_r$ can be constructed. 
Note that $f_r$ is odd due to $m_-=-m_+$.
Then:
\begin{align}
\E_{m^N_T}\big[\, |\nabla F_r|^2\, \big]
&=
\E_{m^N_T}\bigg[ |\nabla F_r|^2\,  {\bf 1}\Big\{\frac{1}{N}\sum_i x_i \in [m_- +\delta, m_+ -\delta] \Big\}\, \bigg]
\nnb
&\leq 
\frac{e^{2 rN}}{c_\delta} e^{-N c_\delta}
\leq 
\frac{1}{c_\delta}
.
\end{align}
On the other hand, the variance reads:
\begin{equation}
\var_{m^N_T}(F_r)
=
\E_{m^N_T}[F_r^2]
\geq 
\E_{m^N_T}[F^2_r {\bf 1}_{B(m_-,\delta)\cup B(m_+,\delta)}]
\geq 
e^{2 rN}\Big(1- \frac{1}{c_\delta}e^{-Nc_\delta}\Big)
.
\end{equation}
The spectral gap of $m^N_T$ is then exponentially small in $N$ as it is smaller than $\E_{m^N_T}[|\nabla F_r|^2]/\var_{m^N_T}(F_r)$.\\

Let us show that, for a suitable $r>0$, 
the test function $F_r$ also gives an exponentially small upper bound on the spectral gap in the graph case. 
Recall the definition $B:= A-d_NP$ and observe  that, by definition of $m^{G_N}_T$,
\begin{equation}
m^{G_N}_T 
\propto
\exp\Big[\frac{(x,Ax)}{2Td_N}\Big]\alpha_V^{\otimes N}(dx)
\propto 
\exp\bigg[\frac{(x,Bx)}{2Td_N}\bigg]
\, m^N_{T}(dx)
.
\label{eq_mGn_as_tilted_mean_field}
\end{equation}
We first obtain useful bounds on exponential moments under $m^N_{T}$. 
Recall our integrability assumption~\eqref{eq_integrability_ZNT}. 
In particular it implies that $\frac{1}{N}\log Z^N_T=O_N(1)$ by Varadhan's lemma, 
see e.g.~\cite{LIU2020503}.  
Assumption~\ref{ass_Vnonquad} on $V$ and the Cauchy-Schwarz inequality then imply that there is $\lambda_0>0$ small enough such that:
\begin{equation}
\limsup_{N\to\infty}\frac{1}{N}\log \E_{m^N_{T}}\Big[ e^{\lambda_0 |x|^2}\Big]
<
\infty
.
\end{equation}
The above and the H\"older inequality imply the existence of $K>0$ such that:
\begin{equation}
\E_{m^N_{T}}\Big[ e^{\frac{ \epsilon_N |x|^2}{T}}\Big]
\leq 
\exp\Big[ \frac{KN\epsilon_N }{T\lambda_0}\Big]
,\qquad 
N\geq 1
.
\end{equation}
Recall that $|(x,Bx)|/d_N\leq \epsilon_N|x|^2$ for each $x\in\R^N$ by Assumption~\ref{ass_graph}. 
Jensen's inequality and the large deviation bound~\eqref{eq_LD_bounds_mNT} for $m^N_T$ also imply the following lower bound: 
for some $C(m_\pm)>0$ depending only on $m_\pm$,
\begin{equation}
\E_{m^N_{T}}\bigg[ \exp\Big[\frac{\pm \epsilon_N|x|^2}{2T}\Big]\, \bigg]
\geq 
\exp\bigg[ -\frac{\E_{m^N_{T}}[x_1^2]N \epsilon_N}{2T}\bigg]
\geq 
\exp\Big[-\frac{C(m_\pm)N\epsilon_N}{2 T}\Big]
.
\end{equation}
We now use these bounds to find $r>0$ such that $F_r$ gives an exponentially small spectral gap for $m^{G_N}_T$. 
Let $r>0$ to be chosen later. 
By Jensen and Cauchy-Schwarz inequalities together with the above exponential moment bounds, 
we find:
\begin{align}
&\E_{m^{G_N}_T}[ |\nabla F_r|^2]
=
\E_{m^{G_N}_T}\bigg[ |\nabla F_r|^2\,  {\bf 1}\Big\{\frac{1}{N}\sum_i x_i\notin [m_- +\delta, m_+ -\delta]
\Big\}\, \bigg]
\nnb
&\ \leq 
e^{2rN} \E_{m^N_{T}}\bigg[{\bf 1}\Big\{\frac{1}{N}\sum_i x_i\notin 
 [m_- +\delta, m_+ -\delta]
\Big\}\,  e^{\frac{\epsilon_N|x|^2}{2T}}\bigg] \E_{m^N_{T}}\Big[ e^{-\frac{\epsilon_N|x|^2}{2T}}\Big]^{-1}
\nnb
&\leq 
e^{2rN} \exp\Big[ \frac{K N\epsilon_N}{ \lambda_0T}\Big]\, \exp\Big[ \frac{C(m_\pm) N\epsilon_N}{2T}\Big]\, 
\P_{m^N_{T}}\bigg[ \frac{1}{N}\sum_i x_i\notin  [m_- +\delta, m_+ -\delta] \bigg]^{1/2}
.
\end{align}
It remains to take $r$ such that $8r<c_\delta$, 
with $c_\delta$ the constant in~\eqref{eq_LD_bounds_mNT} to obtain, for some $C=C(\lambda_0,T,K)$:
\begin{equation}
\E_{m^{G_N}_T}[ |\nabla F_r|^2]
\leq 
e^{-2rN} e^{C\epsilon_N N}
.
\end{equation}

Consider next the variance of $F_r$. 
Note first that $\E_{m^{G_N}_T}[F_r] = 0$  as before,  
as $F_r$ is odd and the measure $m^{G_N}_T$ is symmetric. 
Using the exponential moment bound on the denominator then gives:
\begin{align}
\var_{m^{G_N}_T}(F_r)
&=
\E_{m^{G_N}_T}[F^2_r]
\geq 
e^{2rN} \E_{m^{G_N}_T}\bigg[{\bf 1}\Big\{\frac{1}{N}\sum_i x_i \in B(m_+,\delta)\Big\}\, \bigg]
\nnb
&\geq 
\exp\bigg[2rN - \frac{C(m_\pm)N\epsilon_N}{T} - \frac{KN\epsilon_N}{T\lambda_0}\bigg]
\P_{m^N_{T}} \Big[ \frac{1}{N}\sum_i x_i \in B(m_+,\delta)\Big]
.
\end{align}
Recall that  the probability converges to $1$ by~\eqref{eq_LD_bounds_mNT}. 
 This implies that the spectral gap of $m^{G_N}_T$ is bounded from above by $e^{C\epsilon_N N}e^{-2r N}$ for some $C>0$ independent of the graph. 
\end{proof}

\section{Non-quadratic interaction potential - Proof of Theorem~\ref{thm: nonquadratic mean-field}}
\label{sec; proof theorem 1.7}

In this section, we prove Theorem~\ref{thm: nonquadratic mean-field}. 
We first generalise the notion of the renormalised potential defined in Section~\ref{sec: quadratic}
with respect to the mode decomposition 
and obtain an analogue of Theorem~\ref{thm: nonquadratic mean-field} for temperatures such that  the renormalised potential is strongly convex (see Theorem~\ref{theo_nonquadratic}). 
We then show in Section~\ref{sec_link_renorm_pot_hatFT} that this is equivalent to the condition of Theorem~\ref{thm: nonquadratic mean-field} involving the functional $\cF_T$. 

\subsection{Definition of the renormalised potential}
\label{sec: Definition of the interaction}

	Throughout the section we work with an interaction potential $W=W^+-W^-$ on $\R^d$ satisfying Assumption~\ref{ass_Wnonquad}. 
	Recall in particular definition~\eqref{eq: W- decomposition 0}:
\begin{equation}
\label{eq: W- decomposition}
W^-(x,y) 
:=
\alpha (x,y) + \sum_{k\geq 0}w^-_k n_k(x)n_k(y)
:=
\sum_{k\geq -d}w^-_k n_k(x)n_k(y),
\end{equation}
	where we set $w^-_{-i}=\alpha$ and $n_{-i}(x)=x_i$ for $i\in\{1,\dots,d\}$. 
Furthermore, there are two constants  $M,L>0$, such that the potentials satisfy the bounds :
	\begin{align}
	& \sup_{x,y}|W^+(x,y)|
	\leq 
	M^2, \qquad 
	 \sup_{x,y}| \He W^+(x,y)|_{\rm op} \leq M,
	\nnb
	& \sup_{x,y}\big|W^- (x,y)-\alpha(x,y)\big|
	=
	\Big\|\sum_{k\in\N}w^-_k n_k(x)n_k(y)\Big\|_{\infty}
	\leq
	M^2
	, 	\label{eq: borne sup n}\\
	& \sum_{k\geq -d}w^-_k \sup_{x\in\R^d}|\nabla n_k(x)|^2
	\leq 
	L^2
	. \nonumber	
	\end{align}
Let $\bbH_0 = \bbH_0(W^-)$ denote the Hilbert space:
\begin{equation}
\bbH_0 
=
\bbH_0 (W^-)
:=
\Big\{ (u_k)\in \R^{\N} : \sum_{k\in\N}w^-_k|u_k|^2 < \infty\Big\},
\label{eq_def_Hilbert_space}
\end{equation}
with scalar product  
\begin{equation}
\label{eq: scalar product H0}
(\zeta,\zeta')_{\bbH_0}
:= 
\sum_{k\geq 0} w^-_k \zeta_k\zeta'_k
.
\end{equation}
It will be also convenient to consider the extended space 
\begin{equation}
\bbH 
:=  
\big\{ \psi= ( \varphi,\zeta ):\varphi\in\R^d,\zeta\in\bbH_0\big\} ,
\end{equation}
with scalar product 
\begin{equation}
\label{eq: scalar product H}
(\psi,\psi')_{\bbH}
:= (\varphi,\varphi') + (\zeta,\zeta')_{\bbH_0}
= \sum_{k\geq -d} w^-_k \psi_k\psi'_k
,
\end{equation}
where $(\cdot,\cdot)$ denotes the standard inner product in $\R^d$. 
We always use the letter $\psi$ to denote elements of $\bbH$ and $\psi= (\varphi,\zeta)$, $\varphi\in\R^d$, $\zeta\in\bbH_0$. 
The associated norms are written $\|\cdot\|_{\bbH},\|\cdot\|_{\bbH_0}$.

Using these notations, we introduce the multi-mode counterpart $\mathcal V_T:\bbH\to\R$ of the renormalised potential of Section~\ref{sec: quadratic}.
\begin{definition}
\label{def: general renormalised potential}
For any $\psi \in \bbH$, the renormalised potential is 
\begin{equation}
 \mathcal V_T(\psi)
= \inf_\cM \Big\{ \hat  \cF_T(\cM) + \frac{1}{2 T} \| \psi - \cM \|_{\bbH}^2   \Big\} ,
\label{eq: projection F_T transform}
\end{equation}
where the projection $\hat  \cF_T$ of the mean-field functional was introduced in 
\eqref{eq: projection F_T}. 

\end{definition}

Using Definition~\ref{def: general renormalised potential}, we can now state a condition for the log-Sobolev inequality to hold uniformly in $N$.
\begin{theorem}
\label{theo_nonquadratic}
Let $V$, $W$ satisfying  Assumptions~\ref{ass_Vnonquad} and \ref{ass_Wnonquad}. 
Let $T>0$ be such that $\mathcal V_T$ is $\lambda_T$-strongly convex (in the sense of \eqref{eq: projection F_T convexe}) for some $\lambda_T >0$:  for any $\psi^1, \psi^2$ and $t\in[0,1]$ then
\begin{equation}
t \mathcal V_T(\psi^1) +  (1-t) \mathcal V_T(\psi^2)
\geq  \mathcal V_T \big( \alpha \psi^1 +  (1-\alpha) \psi^2 \big) 
+ \frac{\lambda_T}{2} t(1-t)  \| \psi^1 -  \psi^2 \|_{\bbH}^2
.
\label{eq: V_T convexe}
\end{equation}
Then the measure $m^N_T$ of~\eqref{eq_measure_nonquad} satisfies a log-Sobolev inequality with constant $\gamma^N_{\rm LS}(T)$ bounded below uniformly in $N$. 
\end{theorem}
The following lemma, 
proven in Section~\ref{sec_link_renorm_pot_hatFT},  shows that 
Theorem~\ref{theo_nonquadratic} implies Theorem~\ref{thm: nonquadratic mean-field} as strong convexity of the renormalised potential and of $\hat \cF_T$ are equivalent. 
\begin{lemma}\label{lemma_cvx_cF_is_cvx_cV}
The projected free energy $\hat\cF_T$ is strongly  convex if and only if the renormalised potential $\cV_T$ is strongly convex. 
\end{lemma}

\begin{remark}
\label{Rem: truncation}
Our assumptions allow for an unbounded number of modes in $W^-$. 
It is however enough to prove Theorem \ref{theo_nonquadratic} with a finite number of modes $K$. 
Indeed, define the truncated potential  $W^{-,K}(x,y) 
:= \sum_{k\geq -d}^Kw^-_k n_k(x)n_k(y)$ and let $m_{T,K}^N$ be the approximated mean-field measure.
The truncated mean-field functional reads
\begin{equation}
  \cF_T^{(K)}(\rho) = \cF_T (\rho) +  \sum_{k > K} \frac{w^-_k}{2T}    \Big( \int_{\R^d} n_k(x) \rho(dx) \Big)^2 .
\label{eq: def cFT K}
\end{equation}
As in \eqref{eq: projection F_T}, one can define $\hat \cF_T^{(K)}$ as the projection of $ \cF_T^{(K)}$ on the first $K$ modes.
Strong convexity \eqref{eq: projection F_T convexe} of $\hat \cF_T$ implies strong convexity of $\hat \cF_T^{(K)}$ with the same constant for any $K$. 
Since $(m_{T,K}^N)_{K}$ converges weakly to $m^N_T$,  proving the log-Sobolev inequality for the measure $m_{T,K}^N$  with constant uniform in $K$ and $N$ implies Theorem~\ref{theo_nonquadratic}. 
\end{remark}

\subsection{Proof of Theorem~\ref{theo_nonquadratic}}

The proof of Theorem~\ref{theo_nonquadratic} is split over the following subsections. 
We assume throughout that the number of modes in the decomposition~\eqref{eq: W- decomposition} is finite as explained in Remark~\ref{Rem: truncation}.

\subsubsection{Decomposition of the mean-field measure}

We again rely on the formula for the moment generating function of a Gaussian random variable to decompose the potential $W^-$ in terms of the different modes indexed by the Hilbert space $\bbH$ defined in~\eqref{eq_def_Hilbert_space}. 
Let $\gamma^{\bbH}_{\sigma^2}$ denote the centred Gaussian measure on $\bbH$ with covariance $\sigma^2\id$, 
which formally reads:
\begin{equation}
\label{eq: Gaussian H}
\gamma^{\bbH}_{\sigma^2} (d \psi) \propto \exp \left( - \frac{1}{2 \sigma^2} \| \psi \|_\bbH^2 \right) \prod_{k \geq -d} d \psi_k 
. 
\end{equation}
The formula for the moment generating function of a Gausian random variable then gives:
\begin{align}
\exp\Big( \frac{1}{2TN}\sum_{i,j=1}^N W^-(x_i,x_j)
\Big)
&=
\exp\Big ( \, \frac{1}{2T  N}\Big\|\sum_{i=1}^Nn_\cdot(x_i)\Big\|^2_{\bbH}\, \Big)
\nnb
&=
\E_{\gamma^{\bbH}_{T/N}}\bigg[ \exp \Big( \, \frac{1}{T}\Big(\psi,\sum_{i=1}^N n_\cdot(x_i)\Big)_{\bbH}\, \Big)
\bigg]
,
\end{align}
with $\psi=(\psi_k)_k\in\bbH$ the variable of the measure $\gamma^{\bbH}_{T/N}$. 
The last equation implies the following decomposition of $m^N_T$:
\begin{equation}
\E_{m^N_T}[F]
=
\E_{\nu^r_T} \big[ \E_{\mu^{N,\psi}_T}[F] \big]
,
\end{equation}
where the fluctuation measure $\mu^{N,\psi}_T\in{\bf M}_1((\R^{d})^N)$ is this time not product (particles still interact through $W^+$) and depends on a generalised external field $\psi\in\bbH$:
\begin{equation}
\mu^{N,\psi}_T(dx)
=
e^{N U^N_T(\psi)} \exp\bigg[\, \frac{1}{T}\Big(\psi,\sum_{i=1}^Nn_{\cdot}(x_i)\Big)_{\bbH} -\frac{1}{2T N}\Big\|\sum_{i=1}^Nn_\cdot(x_i)\Big\|^2_{\bbH}\, \bigg]
m^N_T(dx)
.
\label{eq_fluct_measure_nonquad}
\end{equation}
The non-quadratic part $U^N_T(\psi)$ of the renormalised potential, 
now depending on $N$, is given by:
\begin{align}
U^N_T(\psi)
&=
-\frac{1}{N}\log \E_{m^N_T}\bigg [  \exp \Big( \, \frac{1}{T}\Big(\psi,\sum_{i=1}^Nn_{\cdot}(x_i)\Big)_{\bbH} -\frac{1}{2T N}\Big\|\sum_{i=1}^Nn_\cdot(x_i)\Big\|^2_{\bbH}\, \Big) \bigg] \nnb
&=
-\frac{1}{N}\log \E_{\alpha_V^{\otimes N}}\bigg [  \exp \Big( \, \frac{1}{T}\Big(\psi,\sum_{i=1}^Nn_{\cdot}(x_i)\Big)_{\bbH} -\frac{1}{2T  N} \sum_{i,j} W^+(x_i,x_j) \, \Big) \bigg]
+\frac{1}{N}\log Z^N_T
.
\label{eq_def_renormalised_potential_nonquad}
\end{align}
Correspondingly the renormalised measure reads:
\begin{equation}
\nu^r_T(d\psi)
= 
\exp\big( -N U^N_T(\psi)\big) \gamma^{\bbH}_{T/N}(d\psi).
\label{eq_renorm_measure_nonquad}
\end{equation}
Compared with the quadratic case~\eqref{eq_def_V_t_quadratic_W},
 the renormalised potential $\|\psi\|^2_\bbH /(2T) + U^N_T (\psi)$ 
  depends on $N$ and the quadratic terms in $\psi$ are included in the measure $\gamma^{\bbH}_{T/N}$. 
Note also that if $W^+=0$, 
then $\mu^{N,\psi}_T$ is a product measure as in the quadratic case and $U^N_T$ again becomes independent of $N$. 
In the general $W^+\neq 0$ case, 
Proposition~\ref{prop_proximity_VN_Vinfty} below shows that $U^N_T$ is well approximated by its $N\to\infty$ limit $\mathcal U_T(\psi)$ given by   
\begin{align}
\mathcal U_T(\psi)
=
\inf_{\rho\in{\bf  M}_1(\R^d)} 
\bigg\{  H(\rho|\alpha_V) +\frac{1}{2T}&\int_{\R^{d}\times\R^d} W^+(x,y)\rho(dx)\rho(dy) 
\nnb
&\quad -\frac{1}{T}\Big(\psi_\cdot, \int n_\cdot(x)\, \rho(dx)\Big)_{\bbH}\bigg\} 
+\inf \cF_T
,\qquad \psi\in\bbH,
\label{eq_limit_renorm_pot_nonquad}
\end{align}
where  $\cF_T$ is the free energy~\eqref{eq_def_free_energy} and $\alpha_V\propto e^{-V(x)}\, dx$. 
Using the projection over the modes, one gets that 
\begin{equation}
\mathcal U_T(\psi)
= \inf_\cM \Big\{ \hat  \cF_T(\cM) + \frac{1}{2 T} \| \cM \|_{\bbH}^2 -  \frac{1}{T} (\psi,\cM)_\bbH \Big\}
+\inf \cF_T.
\end{equation}
Thus the limiting renormalised potential $\mathcal V_T$ introduced in \eqref{eq: projection F_T transform} can be rewritten as 
\begin{align}
\label{eq: multi-mode counterpart}
\mathcal V_T(\psi)
 =  \frac{1}{2 T} \|\psi\|^2_{\bbH} + \mathcal U_T(\psi).
\end{align}

\subsubsection{The fluctuation measure}
\label{sec_fluct_measure_LSI_nonquad}

We now begin the proof of Theorem~\ref{theo_nonquadratic}. 
Let $W$ satisfy Assumption~\ref{ass_Wnonquad}. 
Recall the definitions~\eqref{eq_fluct_measure_nonquad}--\eqref{eq_renorm_measure_nonquad} of the renormalised measure $\nu^r_T$ and the fluctuation measure $\mu^{N,\psi}_T$, 
built so that the mean-field measure $m^N_T$ of~\eqref{eq_measure_nonquad} decomposes as $m^N_T=\nu^r_T\mu^{N,\psi}_T$. 
As in the quadratic case, this implies the following splitting for the 
entropy of a test function $F:(\R^d)^N\to\R$:
\begin{align}
\ent_{m^N_T}(F^2)
&=
\E_{\nu^r_T}\big[\ent_{\mu^{N,\psi}_{T}}(F^2)\big]
+\ent_{\nu^r_T}\big(\E_{\mu^{N,\psi}_T}[F^2]\big)
.
\label{eq: decomposition entropie}
\end{align}
In this section, we establish a log-Sobolev inequality for the fluctuation measure $\mu^{N,\psi}_T$ with explicit dependence on the field $\psi$. 
The renormalised measure will be studied in Section~\ref{sec_renorm_measure_LSI_nonquad}.  
\begin{proposition}
\label{prop_LSI_fluct_measure_nonquad}
Let $\psi=(\varphi,\zeta)\in\bbH=\R^d\times\bbH_0$ (recall~\eqref{eq_def_Hilbert_space}). 
There are $c,N_0>0$ independent of $N,\psi,T$ such that
the fluctuation measure $\mu^{N,\psi}_T$ satisfies a log-Sobolev inequality with the following constant. 
\begin{itemize}
	\item[(i)] (Theorem 1 in~\cite{wang2024LSI}). For any $N\geq N_0 e^{c\|\zeta\|_{\bbH_0}/T}$, 
\begin{equation}
\big(\gamma^{N,\psi}_{T}\big)^{-1}
\leq 
c\exp\big ( c\|\zeta\|_{\bbH_0}/T\big)
.
\end{equation}
\item[(ii)] For any $N\geq 1$ (recall that $M^2\geq \|W^+\|_\infty, \|W^- -\alpha(x,y)\|_\infty$ by \eqref{eq: borne sup n}),
\begin{equation}
\big(\gamma^{N,\psi}_{T}\big)^{-1}
\leq 
c\exp\Big( \frac{2M}{T}(  MN+2\|\zeta\|_{\bbH_0})\Big)
.
\end{equation}
\end{itemize}
\end{proposition}
Note that the constant in Proposition~\ref{prop_LSI_fluct_measure_nonquad} depends on $\psi=(\varphi,\zeta)\in\bbH$ only through $\zeta$; 
this will be useful in Section~\ref{sec_concl_proof_nonquad}.
\begin{proof}
The claim of item (i) is exactly~\cite[Theorem 1]{wang2024LSI} with explicit dependence on $\zeta$ of the various constants.  
It states the following.
For a flat convex interaction term $W^+$, 
assume that $\He W^+$ has operator norm uniformly bounded by $M>0$; 
that the measure: 
\begin{equation}
\label{eq: mesure intermediaire 1}
{\bf M}_1(\R^d) \ni m^\psi_\rho(dx)
\propto 
\exp\Big[-V(x)+\frac{(\psi,n_\cdot(x))_{\bbH}}{T} -\frac{1}{T}\int W^+(x,y)\rho(dy)\Big]
\, dx
\end{equation}
satisfies a log-Sobolev inequality with constant $\gamma$ uniform in $\rho$; 
and that the one-particle conditional law:
\begin{equation}
\label{eq: mesure intermediaire 2}
\mu^{N,\psi}_{T}(dx_i | (x_j)_{j\neq i})
\propto 
\exp\Big[-V(x_i)+\frac{(\psi,n_\cdot(x_i))_{\bbH}}{T} -\frac{1}{TN}\sum_{j=1}^NW^+(x_i,x_j)\Big]
\, dx_i
,
\end{equation}
has Poincar\'e constant bounded below by the same $\gamma>0$
uniformly in $(x_j)_{j\neq i}$. 
Then by \cite[Theorem 1]{wang2024LSI}, there is a constant $r>0$ depending only on the dimension $d$ such that, 
for any $N>100\max\{ M /\gamma,1\}^3$, 
$\mu^{N,\psi}_T$ satisfies a log-Sobolev inequality with constant $r\min\{\gamma,\gamma^3\}$. \\

The upper bound on $\He W^+$ and flat convexity are implied by Assumption~\ref{ass_Wnonquad}. 
Write $\psi=(\varphi,\zeta)\in\R^d\times\bbH_0$. 
Assumption~\ref{ass_Vnonquad} that $V$ is the sum of a uniformly convex and a Lipschitz  or bounded part 
together with the boundedness of $\int W^+(\cdot,y)\rho(dy)$ for any 
$\rho\in {\bf M}_1(\R^d)$  (see Assumption \ref{ass_Wnonquad}) implies that,  
 for $\zeta =0$, the measures $m^{(\varphi,0)}_\rho$, 
$\mu^{N,(\varphi,0)}_{T,i}(\cdot| (x_j)_{j\neq i})$ satisfy log-Sobolev inequalities with the same constant $\gamma_0 >0$ independent of $\varphi,\rho,(x_j)_{j\neq i}$ 
(e.g.~as a consequence of~\cite[Theorem 1.3]{brigati2024heatflowlogconcavitylipschitz}), 
where we recall that $\varphi$ is the field associated with the quadratic part of $W^-$. 

Consider now the measures $m^{(\varphi,\zeta)}_\rho$, 
$\mu^{N,(\varphi,\zeta)}_{T,i}(\cdot| (x_j)_{j\neq i})$ for  $\zeta \not = 0$.
Assumption~\ref{ass_Wnonquad} gives: 
\begin{equation}
\forall x\in\R^d,\qquad
\big|(\zeta,n_\cdot(x))_{\bbH_0}\big|
\leq 
 \|\zeta\|_{\bbH_0}\sup_{y\in\R^d}\|n_\cdot(y)\|_{\bbH_0}
 \leq 
 \|\zeta\|_{\bbH_0} M
 .
 \label{eq: Cauchy Schwarz H0}
\end{equation}
Thus tilting by $(\zeta,n_\cdot(x))_{\bbH_0}/T$  amounts to a bounded perturbation of the measures $m^{(\varphi,0)}_\rho$  and 
$\mu^{N,(\varphi,0)}_{T,i}(\cdot| (x_j)_{j\neq i})$, 
which deteriorates the log-Sobolev constant by $e^{-4M\|\zeta\|_{\bbH_0}/T}$ at worst by the Holley--Stroock argument. 
 The  log-Sobolev constant  associated with the measures 
\eqref{eq: mesure intermediaire 1} and \eqref{eq: mesure intermediaire 2}
is therefore larger than $e^{-4M\|\zeta\|_{\bbH_0}/T}$ (uniformly in $\gp$), 
and an application of ~\cite[Theorem 1(i)]{wang2024LSI} yields (i).

\medskip

The claim of item (ii) is a simple perturbation argument. 
As $W^+$ is bounded by $M^2$, the Holley--Stroock argument applies and shows:
\begin{equation}
\big(\gamma^{N,\psi}_{T}\big)^{-1}
\leq 
\exp\Big(\frac{2NM^2}{T}\Big) \, \big(\tilde \gamma^{N,\psi}_{T}\big)^{-1}
,
\end{equation}
with $\tilde\gamma^{N,\psi}_T$ the log-Sobolev constant of the product measure:
\begin{equation}
\tilde \mu^{N,\psi}_T(dx)
\propto 
\exp\Big( \frac{1}{T}\Big(\, \psi,\sum_{i=1}^Nn_\cdot(x_i)\Big)_{\bbH}\, \Big) \, \alpha_V^{\otimes N}(dx)
.
\end{equation}
Assumption~\ref{ass_Vnonquad} implies that the probability measure proportional to $e^{\varphi\cdot x}\alpha_V(dx)$ satisfies a log-Sobolev inequality uniformly in $\varphi\in\R^d$.  
On the other hand the contribution $(\zeta,n_\cdot(x))_{\bbH_0}$ of the other fields is bounded as in \eqref{eq: Cauchy Schwarz H0}.  
The Holley--Stroock argument thus gives $(\tilde \mu^{N,\psi}_T)^{-1}\leq ce^{4M\|\zeta\|_{\bbH_0}/T}$, 
which is the claim. 
\end{proof}

\subsubsection{The renormalised measure}\label{sec_renorm_measure_LSI_nonquad}
In this section, 
we establish a log-Sobolev inequality for the renormalised measure for $T>0$ such that the renormalised potential $\cV_T$ is strongly convex  (recall Definition~\eqref{eq: multi-mode counterpart}--\eqref{eq_limit_renorm_pot_nonquad} of $\cV_T$). 
Our aim is to show that 
the renormalised potential is close enough to its limit $\cV_T(\varphi)$ for the uniform convexity of the latter to imply the log-Sobolev inequality at fixed $N$. 
We prove the following result. 
\begin{proposition}
\label{prop: LSI renormalised measure}
If $\mathcal V_T$ is $\lambda_T$-strongly convex~\eqref{eq: V_T convexe} for some $\lambda_T>0$, 
then
there is a constant $C>0$ such that uniformly in $N$
\begin{equation}
\ent_{\nu^r_T}\big(\E_{\mu^{N,\psi}_T}[F^2]\big)
\leq 
\frac{ 2 C L^2 }{\lambda_T } 
\E_{\nu^r_T} \Big[\, \frac{1}{\big(\gamma^{N,\psi}_T\big)^{2}} \,
 \E_{\mu^{N,\psi}_T}\big[|\nabla F|^2\big] \Big]
.
\label{eq_bound_ent_renorm_measure_nonquad}
\end{equation}
Recall from~\eqref{eq: borne sup n} that $L^2=\|\sup_x |\nabla n_\cdot(x)|\|^2_{\bbH}$.
\end{proposition}
Proposition \ref{prop: LSI renormalised measure} makes use of the following estimate that shows that the renormalised potential $U^N_T$ is very close to its limit $\cU_T$ defined in ~\eqref{eq_limit_renorm_pot_nonquad}. 
\begin{proposition}\label{prop_proximity_VN_Vinfty}
There is $C>0$ such that, for all $\psi\in\bbH$ and all $N\geq 1$:
\begin{equation}
\big|U^N_T(\psi) - \cU_T(\psi)\big|
\leq 
\frac{C}{N}
.
\end{equation}
\end{proposition}
We will also need the following lemma that states that the Bakry--\'Emery criterion can be used to obtain a log-Sobolev inequality on $\bbH$. 
\begin{lemma}\label{lemm_LSI_renorm_nonquad}
Let $\lambda_T>0$ and assume that $\mathcal V_T = \cU_T + \|\cdot\|^2_{\bbH}/2T$ is $\lambda_T$-convex (recall~\eqref{eq: V_T convexe}). 
Then there is $C>0$ such that  the renormalised measure satisfies the followint log-Sobolev inequality: 
\begin{equation}
\ent_{\nu^r_T}\big(F^2 \big)
\leq 
\frac{2 \, C}{N \lambda_T }\E_{\nu^r_T}\Big[ \sum_{k \geq -d} (w_k^-)^{-1}
\Big( \partial_{\psi_k} F \Big)^2 \Big].
\label{eq_LSI_renormalised measure W}
\end{equation}
\end{lemma}
Assuming Proposition~\ref{prop_proximity_VN_Vinfty} and Lemma~\ref{lemm_LSI_renorm_nonquad}, 
let us prove~Proposition \ref{prop: LSI renormalised measure}.

\begin{proof}[Proof of Proposition \ref{prop: LSI renormalised measure}]
Lemma~\ref{lemm_LSI_renorm_nonquad} gives 
\begin{equation}
\ent_{\nu^r_T}\big(\E_{\mu^{N,\psi}_T}[F^2]\big)
\leq 
\frac{2 \, C}{N \lambda_T}\E_{\nu^r_T}\Big[ \sum_{k \geq -d} (w_k^-)^{-1}
\Big( \partial_{\psi_k}\sqrt{\E_{\mu^{N,\psi}_T}[F^2]}\Big)^2 \Big].
\label{eq_LSI_withHgradient}
\end{equation}
Expanding the gradient yields:
\begin{align}
\sum_{k \geq -d} (w_k^-)^{-1}
& \Big( \partial_{\psi_k}\sqrt{\E_{\mu^{N,\psi}_T}[F^2]}\Big)^2
=
\frac{1}{4\E_{\mu^{N,\psi}_T}[F^2]} \sum_{k\geq -d} (w_k^-)^{-1}|\partial_{\psi_k}\E_{\mu^{N,\psi}_T}[F^2]|^2
\nnb
&=
\frac{1}{4T^2\E_{\mu^{N,\psi}_T}[F^2]} \sum_{k\geq-d} w_k^-\Big|\cov_{\mu^{N,\psi}_T}\Big(F^2,\sum_{i=1}^Nn_k(x_i)\Big)\Big|^2
\nnb
& \leq
\frac{4}{\big(\gamma^{N,\psi}_T \big)^2} 
\frac{1}{4T^2\E_{\mu^{N,\psi}_T}[F^2]} \sum_{k\geq-d} w_k^-
 \big( N \sup_x | \nabla n_k (x)|^2 \big)   \E_{\mu^{N,\psi}_T}[F^2] \, \E_{\mu^{N,\psi}_T}\big[|\nabla F|^2\big] \nnb
& \leq
\frac{ N L^2 }{T^2  \big(\gamma^{N,\psi}_T \big)^2}  \E_{\mu^{N,\psi}_T}\big[|\nabla F|^2\big] 
,
\end{align}
where we used  Lemma~\ref{lemm_BH} 
(denoting by $\gamma^{N,\psi}_T$ the log-Sobolev constant of the measure 
$\mu^{N,\psi}_T$) 
and the  assumed bound $L^2\geq \|\sup_x |\nabla n_\cdot(x)|\|^2_{\bbH}$ in the last line (recall~\eqref{eq: borne sup n}). 
Combined with \eqref{eq_LSI_withHgradient},
this completes the proof of Proposition \ref{prop: LSI renormalised measure}.
\end{proof}

\bigskip

\begin{proof}[Proof of Proposition~\ref{prop_proximity_VN_Vinfty}]
Let $\psi\in\bbH$ and denote by $\cF_T^\psi(\cdot)$ the functional appearing in the variational principle \eqref{eq_limit_renorm_pot_nonquad} defining $\cU_T$, which we rewrite as:
\begin{equation}
\cF_T^\psi(\rho)
=
H(\rho|\alpha^\psi_{V,T}) 
+ \frac{1}{2T} \int_{\R^{d}\times\R^d} W^+ d\rho^{\otimes 2}
+C(\psi,W,V,T)
,\qquad 
\rho\in {\bf M}_1(\R^d)
,
\end{equation}
where $C(\psi,W,V,T)>0$ is a constant and:
\begin{equation}
{\bf M}_1(\R^d)\ni \alpha^\psi_{V,T}(dx)
=
\frac{1}{Z^\psi_{V,T}} 
\exp\Big(-V(x)+\frac{1}{T}(\psi,n_\cdot(x))_{\bbH}\Big)\, dx
.
\label{eq_def_alpha_psiVT}
\end{equation}
As $\rho\mapsto \int W^+d\rho^{\otimes 2}$ is bounded below, convex and $\rho\mapsto H(\rho|\alpha^\psi_{V,T})\geq 0$ is (strictly) convex, 
the functional $\cF^\psi_T$ admits a (unique) minimiser,  
call it $\mu^{\infty,\psi}_T$. 
This minimiser is absolutely continuous with respect to $\alpha^{\psi}_{V,T}$. 
The uniqueness will not be used below.

The critical point equation for $\cF_T^\psi$ gives the following identities for $\mu^{\infty,\psi}_T$. 
Letting $f := \frac{{\rm d \mu^{\infty,\psi}_T}}{{\rm d\alpha_{V,T}^\psi}}$, 
for some constant $C=C(\psi)>0$,
\begin{equation}
\log f(x) 
= 
-\frac{1}{T}\int W^+(x,y)\,\mu^{\infty,\psi}_T(dy)
+  C(\psi)
\qquad
\text{for $\mu^{\infty,\psi}_T$-a.e. $x$}.
\label{eq_fixed_point_Fpsi}
\end{equation}
For once, we compute the constant $C(\psi)$ as we are looking to compensate it precisely. 
Notice the elementary identity:
\begin{equation}
H(\rho|\alpha_{V}) 
= 
H(\rho|\alpha^\psi_{V,T})
+
\frac{1}{T}\int (\psi,n_\cdot(x))_{\bbH}\, \rho(dx) - \log Z^\psi_{V,T},
\qquad 
\rho\in {\bf M}_1(\R^d)
.
\end{equation}
Following~\cite[Proposition 4.2 item (3)]{MR743526},
we 
integrate both sides of~\eqref{eq_fixed_point_Fpsi} against $\mu^{\infty,\psi}_T$, 
recalling the definition~\eqref{eq_limit_renorm_pot_nonquad} of $\cU_T$ 
to obtain:
\begin{align}
&H(\mu^{\infty,\psi}_T|\alpha^\psi_{V,T})
=
-\frac{1}{T}\int W^+(x,y)\, \mu^{\infty,\psi}_T(dx)\,\mu^{\infty,\psi}_T(dy)
+ C (\psi)
\nnb
&\qquad=
\cU_T(\psi) - \frac{1}{2T}\int W^+(x,y)\, \mu^{\infty,\psi}_T(dx)\,\mu^{\infty,\psi}_T(dy) + \log Z^\psi_{V,T}- \inf \cF_T
,
\end{align}
so that:
\begin{equation}
C (\psi)
=
\cU_T(\psi) + \frac{1}{2T}\int W^+(x,y)\, \mu^{\infty,\psi}_T(dx)\,\mu^{\infty,\psi}_T(dy)
 + \log Z^\psi_{V,T}- \inf \cF_T
.
\end{equation}
Recall now that the renormalised potential $U^N_T(\psi)$ \eqref{eq_def_renormalised_potential_nonquad} reads:
\begin{equation}
e^{-N U^N_T(\psi)}
= 
 \frac{ ( Z^\psi_{V,T})^N }{Z^N_T}\int \exp\Big( -\frac{N}{2T}\int W^+(x,y)\,  \pi^N(dx)\pi^N(dy) \Big) \, d(\alpha_{V,T}^\psi)^{\otimes N}
,
\end{equation}
where $\pi^N=\frac{1}{N}\sum_i \delta_{x_i}$ denotes the empirical measure. 
Turning $\alpha^{\psi}_{V,T}$ into $\mu^{\infty,\psi}_T$, 
we get:
\begin{align}
\cU_T(\psi)- U^N_T(\psi)
=
\frac{1}{N}\log\int \exp\Big( -\frac{N}{2T}\int W^+\,  (\pi^N-\mu^{\infty,\psi}_T)^{\otimes 2}\Big) \, d(\mu^{\infty,\psi}_T)^{\otimes N}
.
\end{align}
It therefore suffices to prove that the integral is bounded by $O(1)$ uniformly in $\psi$. 
The flat convexity of $W^+$ implies that  
for any  finite signed measure $\rho$ such that $\int \rho =0$ one has 
$\int W^+ \rho^{\otimes 2} \geq 0$, thus
the exponential is at most $1$:
\begin{equation}
\cU_T(\psi)- U^N_T(\psi)
\leq 
0
.
\end{equation}
On the other hand, 
Jensen's inequality and an expansion of $(\pi^N-\mu^{\infty,\psi}_T)^{\otimes 2}$ give:
\begin{align}
\cU_T(\psi)- U^N_T(\psi)
&\geq 
-\frac{1}{2T}\E_{(\mu^{\infty,\psi}_T)^{\otimes N}}\Big[\int W^+\,  (\pi^N-\mu^{\infty,\psi}_T)^{\otimes 2}\Big]
\nnb
&=
\frac{1}{2TN}\int W^+\,  (\mu^{\infty,\psi}_T)^{\otimes 2}
-\frac{1}{2TN}\int W^+(x,x)\, \mu^{\infty,\psi}_T(dx)
.
\end{align}
As $W^+$ is bounded, 
the right-hand side above is bounded uniformly in $\psi$ by $C/N$ for some $C>0$. 
This concludes the proof.
\end{proof}
\begin{proof}[Proof of Lemma~\ref{lemm_LSI_renorm_nonquad}]
Recall from Remark~\ref{Rem: truncation} that we work under the assumption that $\bbH$ is finite-dimensional.  
By assumption $\cV_T = \cU_T + \|\cdot\|^2_{\bbH}/2T$ is $\lambda_T$-convex. 
It is proven in~\cite[Theorem 2]{Greene1976CCF} (see~\cite{Azagra2012GlobalAF} for a claim directly applicable to the present setting) 
that there is a sequence $\cV^{(n)}_T:\bbH\to\R$ of $\lambda_T$-convex $C^2$ functions such that $\|\cV_T-\cV^{(n)}_T\|_\infty\leq 2^{-n}$. 
The Bakry--\'Emery criterion then implies that the probability measure with density proportional to $e^{\cV^{(n)}_T}$ satisfies a log-Sobolev inequality 
of the form \eqref{eq_LSI_renormalised measure W}. 
By weak convergence as $n\to\infty$ the same is therefore true for $e^{-N\cV_T(\psi)}\, d\psi$. 
Since $\|U^N_T-\cU_T\|_\infty\leq C/N$ by Proposition~\ref{prop_proximity_VN_Vinfty},  
another application of the Holley--Stroock result concludes the proof 
up to changing the log-Sobolev constant by a multiplicative factor.
\end{proof}

\subsubsection{Conclusion of the proof of Theorem~\ref{theo_nonquadratic}}
\label{sec_concl_proof_nonquad}

At this point we have established in \eqref{eq: decomposition entropie} and \eqref{eq_bound_ent_renorm_measure_nonquad} that if $\mathcal V_T$ is $\lambda_T$-strongly convex  for some $\lambda_T>0$:
\begin{align}
\ent_{m^N_T}(F^2)
&\leq \frac{2 C L^2 }{\lambda_T} 
\E_{\nu^r_T} \Big[\, \frac{1}{\big(\gamma^{N,\psi}_T\big)^{2}} \,
 \E_{\mu^{N,\psi}_T}\big[|\nabla F|^2\big] \Big]
 +
\E_{\nu^r_T} \Big[\,  \ent_{\mu^{N,\psi}_T}(F^2)\, \Big] 
.
\end{align}
Using again the log-Sobolev inequality for the fluctuation measure (Proposition \ref{prop_LSI_fluct_measure_nonquad}), this implies:
\begin{align}
\ent_{m^N_T}(F^2)
\leq 
\E_{\nu^r_T} \bigg[\, \Big(\frac{2  C L^2}{\lambda_T (\gamma^{N,\psi}_T)^{2}} + \frac{2}{\gamma^{N,\psi}_T} \Big)\E_{\mu^{N,\psi}_T}\big[|\nabla F|^2\big] \, \bigg]  .
\label{eq_ccl_proof_non_quad0}
\end{align}
It remains to express the right-hand side in terms of the Dirichlet form for the measure $m^N_T$. 
The starting point is the following elementary Gaussian identity. 
\begin{lemma}\label{lemm_gaussian}
For any $G:(\R^{d})^N\to\R_+$ and $\Phi:\bbH\to\R_+$,
\begin{equation}
\E_{\nu^r_T}\Big[ \Phi(\psi)\E_{\mu^{N,\psi}_T}[G(x)]\Big]
=
\E_{m^{N}_T}\Big[G(x)\,  \E_{\gamma^{\bbH,x}_{T/N}}[\Phi(\psi)]\Big]
,
\end{equation}
where $\gamma^{\bbH,x}_{T/N}$ is the Gaussian measure on $\bbH$ with variance $T/N$ (as in \eqref{eq: Gaussian H}) and mean 
$\big( \frac{1}{N}\sum_i n_k(x_i) \big)_{k \geq -d}$. 
\end{lemma}
\begin{proof}
We go back to the definition~\eqref{eq_fluct_measure_nonquad}--\eqref{eq_renorm_measure_nonquad} of the decomposition $m^N_T=\nu^r_T\mu^{N,\psi}_T$ in terms of moment generating function and exchange the order of integration, which is legitimate as both $G,\Phi$ are non-negative:
\begin{align}
\E_{\nu^r_T}\Big[ \Phi(\psi)\E_{\mu^{N,\psi}_T}[G(x)]\Big]
&=
\E_{\gamma^{\bbH}_{T/N}}\bigg[ \Phi(\psi) \E_{m^N_T}\bigg[G \exp\Big[\, \frac{1}{T}\Big(\psi,\sum_i n_{\cdot}(x_i)\Big)_{\bbH} - \frac{N}{2T}\Big\| \frac{1}{N}\sum_{i=1}^N n_{\cdot}(x_i)\Big\|^2_{\bbH}\, \Big]\bigg]\bigg]
\nnb
&=
\E_{m^{N}_T}\Big[G(x)\,  \E_{\gamma^{\bbH,x}_{T/N}}[\Phi(\psi)]\Big]
.
\end{align}
\end{proof}
\begin{proof}[Proof of Theorem~\ref{theo_nonquadratic}]
Let $A>0$ to be chosen later. 
Write for short:
\begin{equation}
\kappa^{N,\psi}_T
:=
\frac{2 C L^2}{\lambda_T (\gamma^{N,\psi}_T)^{2}} + \frac{2}{\gamma^{N,\psi}_T}.
\end{equation}
Split the expectation on $\psi=(\varphi,\zeta)\in\bbH= \R^d\times\bbH_0$ in the right-hand side of~\eqref{eq_ccl_proof_non_quad0} as follows: 
\begin{align}
\E_{\nu^r_T}\Big[\, \kappa^{N,\psi}_T\E_{\mu^{N,\psi}_T}\big[|\nabla F|^2\big]\, \Big]
&= 
\E_{\nu^r_T}\Big[\,{\bf 1}_{\|\zeta\|_{\bbH_0}\leq A}  \,  \kappa^{N,\psi}_T  \,  \E_{\mu^{N,\psi}_T}\big[|\nabla F|^2\big]\, \Big]
\nnb
&\quad +
\E_{\nu^r_T}\Big[ \, {\bf 1}_{\|\zeta\|_{\bbH_0}>A} \, \kappa^{N,\psi}_T  \, \E_{\mu^{N,\psi}_T}\big[|\nabla F|^2\big]\, \Big]
.
\label{eq_ccl_proof_non_quad1}
\end{align}
Consider first the case where $\|\zeta\|_{\bbH_0}\leq A$. 
By Proposition~\ref{prop_LSI_fluct_measure_nonquad} item (i) one has then, 
for some $c,N_0>0$ independent of $N,A,\psi$ and all $N\geq N_0 e^{cA/T}$:
\begin{align}
\E_{\nu^r_T}\Big[\,{\bf 1}_{\|\zeta\|_{\bbH_0}\leq A} \, \kappa^{N,\psi}_T\E_{\mu^{N,\psi}_T}\big[\, |\nabla F|^2\, \big]\, \Big]
&\leq 
ce^{cA/T}\Big(\frac{1}{\lambda_T }+1\Big) \E_{\nu^r_T}\Big[ \E_{\mu^{N,\psi}_T}\big[\, |\nabla F|^2\, \big]\Big]
\nnb
&=
ce^{cA/T}\Big(\frac{1}{\lambda_T }+1\Big) \E_{m^N_T}\big[\, |\nabla F|^2\, \big]
.
\end{align}
Proposition~\ref{prop_LSI_fluct_measure_nonquad} item (ii) implies on the other hand that, for $N\leq N_0 e^{cA/T}$ and some $c'>0$:
\begin{equation}
\E_{\nu^r_T}\Big[\,{\bf 1}_{\|\zeta\|_{\bbH_0}\leq A} \, \kappa^{N,\psi}_T\, \E_{\mu^{N,\psi}_T}\big[\, |\nabla F|^2\, \big]\, \Big]
\leq 
c'e^{c'e^{cA/T}}\Big(\frac{1}{\lambda_T}+1\Big) \E_{m^N_T}\big[\, |\nabla F|^2\, \big]
.
\label{eq_bound_first_half_LSI_nonquad}
\end{equation}
The first term in the right-hand side of~\eqref{eq_ccl_proof_non_quad1} is thus bounded by $C(T,A)(1+1/\lambda_T)$ uniformly in $N$ for a locally bounded $T'\mapsto C(T',A)\geq 0$.\\

Consider now the second term in the right-hand side of~\eqref{eq_ccl_proof_non_quad1}. 
Using Lemma~\ref{lemm_gaussian} with $G=|\nabla F|^2\geq 0$ and $\Phi = {\bf 1}_{\|\zeta\|_{\bbH_0}>A} \,  \kappa^{N,\psi}_T$ yields:
\begin{align}
&\E_{\nu^r_T}\Big[\,{\bf 1}_{\|\zeta\|_{\bbH_0}>A} \,  \kappa^{N,\psi}_T\, \E_{\mu^{N,\psi}_T}\big[|\nabla F|^2\big]\, \Big]
=
\E_{m^N_T}\Big[\,|\nabla F|^2\, \E_{\gamma^{\bbH,x}_{T/N}}\big[\, {\bf 1}_{\|\zeta\|_{\bbH_0}>A} \, \kappa^{N,\psi}_T\, \big] \, \Big]
\nnb
&\hspace{3cm}
\leq
c e^{2M^2 N/T} \, \E_{m^N_T}\bigg[\,|\nabla F|^2\,  \E_{\gamma^{\bbH,x}_{T/N}}\Big[ e^{4M\|\zeta\|_{{\bbH}_0}/T} {\bf 1}_{\|\zeta\|_{\bbH_0}>A } \, \Big]
\, \bigg]
,
\end{align}
where we use item (ii) of Proposition~\ref{prop_LSI_fluct_measure_nonquad} to get the last line.  
Concentration under the Gaussian measure $\gamma^{\bbH,x}$ gives a bound on the probability of the event $\{\| \zeta \|_{\bbH_0}>A\}$. 
Since $\sup_x\|n_\cdot(x)\|_{\bbH_0}\leq M$ by Assumption~\ref{ass_Wnonquad}, 
for each $A>2M$ it holds that $\|\zeta-\frac{1}{N}\sum_i n_\cdot(x_i)\|_{\bbH_0}>A/2$ if $\|\zeta\|_{\bbH_0}>A$. 
Thus, for some $c'>0$ and each $A>2M$:
\begin{equation}
\bbE_{\gamma^{\bbH,x}_{T/N}} \big[ \|\zeta\|_{\bbH_0}> A \big]
\leq 
c'\exp\Big[-\frac{N  A^2}{8T}\Big]
.
\label{eq: decroissance gaussienne}
\end{equation}
We can then write, for some $c''>0$:
\begin{align}
\E_{\gamma^{\bbH,x}_{T/N}}\Big[ e^{4M\|\zeta\|_{{\bbH}_0}/T}{\bf 1}_{\|\zeta\|_{\bbH_0}>A } \, \Big]
&\leq 
c' \int_{A}^\infty 
\exp\Big[\frac{4Ma}{T}-\frac{N  a^2}{8T}\Big]\, da
\nnb
&\leq 
c' \exp\Big[ \frac{32 M^2}{NT}-\frac{c''N}{T}\Big(A-\frac{16M}{N}\Big)^2\Big]
.
\end{align}
It remains to take $A>2M$ so that also $32M^2/(NT)-c''NA^2/4T+2NM^2T\leq 0$ for each $N\geq 1$, 
which is possible for $A$ larger than a constant depending only on $M,T$, to obtain:
\begin{equation}
\E_{\nu^r_T}\Big[\,{\bf 1}_{\|\zeta\|_{\bbH_0}>A}\,  \kappa^{N,\psi}_T\, \E_{\mu^{N,\psi}_T}\big[|\nabla F|^2\big]\, \Big]
\leq 
c' \E_{m^N_T}\big[\, |\nabla F|^2\, \big]
.
\end{equation}
Together with~\eqref{eq_bound_first_half_LSI_nonquad} this concludes the proof.

\end{proof}

\subsection{Proof of Lemma~\ref{lemma_cvx_cF_is_cvx_cV}}\label{sec_link_renorm_pot_hatFT}

\begin{proof}[Proof of Lemma~\ref{lemma_cvx_cF_is_cvx_cV}]
Recall Definition~\ref{def: general renormalised potential} of $\cV_T$. 
Assume $\hat\cF_T$ is $\delta$-strongly convex ($\delta>0$). 
Let $\psi_1,\psi_2\in\bbH$ and $\alpha\in[0,1]$. 
Write $\psi_\alpha = \alpha \psi_1+(1-\alpha)\psi_2$, 
$\Delta\psi=\psi_1-\psi_2$ and similarly define $\cM_\alpha$, $\Delta\cM$ for $\cM_1,\cM_2\in\bbH$. 
Using the strong convexity of $\hat\cF_T$ in the inequality and, 
in the last line, 
the fact that $1$-strong convexity holds for $\|\cdot\|^2_{\bbH}/2$  with an equal sign, 
\begin{align}
\alpha&\cV_T(\psi_1)+(1-\alpha)\cV_T(\psi_2)
\nnb
&=
\inf_{\cM_1,\cM_2}\Big\{\alpha\hat\cF_T(\cM_1)+(1-\alpha)\hat\cF_T(\cM_2) + \frac{\alpha}{2T}\|\psi_1-\cM_1\|_{\bbH}^2+\frac{1-\alpha}{2T}\|\psi_2-\cM_2\|_{\bbH}^2\Big\}
\nnb
&\geq 
\inf_{\cM_1,\cM_2}\Big\{\cF_T(\cM_\alpha) + \frac{\delta\alpha(1-\alpha)}{2}\|\Delta\cM\|^2_{\bbH}
+ \frac{\alpha}{2T}\|\psi_1-\cM_1\|_{\bbH}^2+\frac{1-\alpha}{2T}\|\psi_2-\cM_2\|_{\bbH}^2\Big\}
\nnb
&=
\inf_{\cM_1,\cM_2}\Big\{\cF_T(\cM_\alpha) + \frac{\delta\alpha(1-\alpha)}{2}\|\Delta\cM\|^2_{\bbH}
+ \frac{1}{2T}\|\psi_\alpha-\cM_\alpha\|_{\bbH}^2 +\frac{\alpha(1-\alpha)}{2T}\|\Delta\psi-\Delta\cM\|_{\bbH}^2\Big\}
.
\end{align}
Changing variables from $\cM_1,\cM_2$ to $\cM=\cM_\alpha$, $\cM'=\Delta\cM$ yields:
\begin{align}
\alpha\cV_T(\psi_1)+(1-\alpha)\cV_T(\psi_2)
&\geq 
\cV_T(\psi_\alpha) 
+\inf_{\cM'}\Big\{\frac{\delta\alpha(1-\alpha)}{2}\|\cM'\|^2_{\bbH}+\frac{\alpha(1-\alpha)}{2T}\|\Delta\psi-\cM'\|_{\bbH}^2\Big\}
\nnb
&=
\cV_T(\psi_\alpha) + \frac{\alpha(1-\alpha)}{2}\frac{\delta}{\delta T+1}\|\psi_1-\psi_2\|^2_{\bbH}
.
\end{align}
Thus $\cV_T$ is strongly convex.

\medskip

Conversely, we claim that
$\hat\cF_T$ can be defined in terms of the renormalised potential as follows:
\begin{equation}
\hat\cF_T(\cM)
=
\sup_{\psi\in\bbH}\Big\{ \cV_T(\psi) - \frac{1}{2T}\|\psi-\cM\|_{\bbH}^2\Big\} - \inf\cF_T
.\label{eq_hatcF_T_as_supremum}
\end{equation}
If this is true then an identical proof gives that strong convexity of $\cV_T$ implies strong convexity of $\hat\cF_T$. 
Let us thus show~\eqref{eq_hatcF_T_as_supremum}. 
Recall that ${\bf P}(\cM)=\{\rho: \int_{\R^d} n_\cdot(x)\, \rho(dx) = \cM\}$ for $\cM\in\bbH$.   
Define the (strictly) convex part $\cG_T$ of the free energy $\cF_T$:
\begin{equation}
\cG_T(\rho)
=
\cF_T(\rho)+ \frac{1}{2T}\Big\|\int_{\R^d} n_\cdot(x)\, \rho(dx)\Big\|_{\bbH}^2
.
\end{equation}
Define also:
\begin{equation}
\hat \cG_T(\cM)
=
\inf_{\rho\in {\bf P}(\cM)}\cG_T(\rho)
=
\hat\cF_T(\cM) + \frac{\|\cM\|_{\bbH}^2}{2T},
\end{equation}
with $\hat \cG_T(\cM) = + \infty$ if ${\bf P}(\cM) = \emptyset$.
Then:
\begin{equation}
\cV_T(\psi) 
=
\frac{\|\psi\|^2_{\bbH}}{2T}+\inf_{\cM\in\bbH} \Big\{\hat \cG_T(\cM) - \frac{1}{T}(\psi,\cM)_{\bbH}\Big\} + \inf\cF_T
.
\end{equation}
In particular $-\cV_T (T\cdot)+\frac{T\|\cdot\|^2_{\bbH}}{2}$ is the Legendre transform of $\hat\cG_T$. 
As $\cG_T$ is convex and lower semi-continuous, 
so is $\hat\cG_T$.  
In addition $\hat\cG_T$ is finite as soon as ${\bf P}(\cM)$ is not empty. 
The Legendre transform can therefore be inverted~\cite[Theorem 1.11]{Brezis}:
\begin{equation}
\hat\cG_T(\cM)
=
\sup_{\psi\in\bbH}\Big\{(\psi,\cM)_{\bbH}-\Big(-\cV_T (T\psi)+\frac{T\|\psi\|^2_{\bbH}}{2} +\inf\cF_T\Big)\Big\}.
\label{eq_cG_as_Legendre_transform}
\end{equation}
Since $\hat\cG_T(\cM)=\hat\cF_T(\cM)+\|\cM\|_{\bbH}^2/(2T)$, 
this yields~\eqref{eq_hatcF_T_as_supremum}:
\begin{equation}
\hat\cF_T(\cM)
=
\sup_{\psi\in\bbH}\Big\{\frac{1}{T}(\psi,\cM)_{\bbH}+ \cV_T(\psi)-\frac{\|\psi\|^2_{\bbH}}{2T} -\frac{\|\cM\|^2_{\bbH}}{2T}\Big\}
-\inf\cF_T
.
\end{equation}
%
%
\end{proof}

 \appendix
  
\section{Proof of $L^\infty$ covariance bound}
The following statement is proven e.g. in~\cite[Lemma 5]{MR2291434} which goes back to \cite{MR1715549,MR1704666}. 
We prove a slightly stronger form below in \eqref{e:F2}. 
\begin{lemma}[Lemma 5 in~\cite{MR2291434}]
\label{lemm_BH}
Let $N \geq 1$ and $\mu$ be a probability measure on $\R^N$ satisfying a log-Sobolev inequality with constant $\gamma_{\rm LS}>0$. 
Take $F:\R^N\to\R$ to be smooth and compactly supported and a Lipschitz function $H:\R^N\to\R$.
Then:
\begin{equation} 
\label{e:F1}
\cov_{\mu}(F^2,H)^2
\leq 
\frac{4}{\gamma_{\rm LS}^2}  \sup_{x\in\R^N}|\nabla H (x)|^2 \, 
\E_{\mu}[F^2] \; \E_{\mu}\big[|\nabla F|^2\big],
\end{equation}
with the notation $|\nabla H (x)| = \big( \sum_{i =1}^N (\partial_{x_i}  H (x))^2 \big)^{1/2}$.
The same bound holds for a vector valued function $H(x)\in\R^d$.

Let $H_i:\R\to\R$ be Lipschitz ($1\leq i\leq N$) and let $F:\R^N\to\R$ be smooth and compactly supported. Then:
\begin{equation}
\label{e:F2}
\sum_{i=1}^N\cov_{\mu}(F^2,H_i(x_i))^2
  \leq 
  \frac{4}{\gamma_{\rm LS}^2} \max_{1\leq i\leq N}\|H'_i\|^2_\infty \;  \E_{\mu}[F^2] \;  \E_{\mu}\big[|\nabla F
  |^2\big] 
  .
\end{equation} 
\end{lemma}
\begin{proof}
The statement \eqref{e:F1} in~\cite[Lemma 5]{MR2291434} only concerns real-valued $H$, 
but the claim with vector-valued $H$ with independent components is straightforward from the proof. 
We follow the same method  to prove \eqref{e:F2}.

Without loss of generality assume $\E_\mu[F]=1$. 
Let $(P_t)_{t\geq 0}$ denote the semi-group associated with the Langevin dynamics and let $L$ denote the corresponding generator. 
Then:
\begin{align}
\cov_\mu(F,H_i) 
&= 
\int_{\R^N} H_i(x_i) \big(P_0F(x) - P_\infty F(x)\big)\,d\mu (x)
=
- \int_0^\infty \int_{\R^N} H_i(x_i) LP_t F(x)\,d\mu(x)\, dt \nonumber\\
&=
\int_0^\infty \int_{\R^N} H'_i(x_i)\partial_{x_i}P_t F(x)\,d\mu(x)\, dt
\nnb
&\leq
\|H'_i\|_\infty\int_0^\infty \Big(\int_{\R^N} \frac{\big(\partial_{x_i}P_t F\big)^2}{P_tF}\,d\mu\Big)^{1/2}\, dt
,
\label{eq_CS_covariance_bound}
\end{align}
where we used the integration by parts formula $\E_{\mu}[FLG] = - \E_{\mu}[(\nabla F,\nabla G)]$ in the second equality, 
and Cauchy-Schwarz inequality with $\E_{\mu}[P_t F] = 1$ in the  inequality. 
Thus:
\begin{equation}
\sum_{i=1}^N\cov_\mu(F,H_i)^2
\leq
\max_{1\leq i\leq N}\|H'_i\|_\infty^2 \;
\sum_{i=1}^N\bigg[\int_0^\infty \Big(\int \frac{\big(\partial_{x_i}P_t F\big)^2}{P_tF}\,d\mu\Big)^{1/2}\, dt\bigg]^2
.
\end{equation}
To have the square go inside the time integral, 
let $\epsilon>0$ to be chosen later and write:
\begin{align}
\bigg[\int_0^\infty \Big(\int \frac{\big(\partial_{x_i}P_t F\big)^2}{P_tF}\,d\mu\Big)^{1/2}\, dt\bigg]^2
&=
\frac{1}{\epsilon^2}\bigg[\int_0^\infty \epsilon e^{-\epsilon t}\Big(e^{2\epsilon t}\int \frac{\big(\partial_{x_i}P_t F\big)^2}{P_tF}\,d\mu\Big)^{1/2}\, dt\bigg]^2\nonumber\\
&\leq 
\frac{1}{\epsilon^2}\int_0^\infty \epsilon e^{-\epsilon t}e^{2\epsilon t}\int \frac{\big(\partial_{x_i}P_t F\big)^2}{P_tF}\,d\mu \, dt
.
\end{align}
Summing over all $1\leq i\leq N$ yields:
\begin{equation}\label{eq_cov_continuous_0}
\sum_{i=1}^N\cov_\mu(F,H_i)^2
\leq \max_{1\leq i\leq N}\|H'_i\|_\infty^2 \;
\frac{1}{\epsilon^2}\int_0^\infty  dt \, \epsilon e^{\epsilon t} \,
\E_{\mu} \left[ \frac{|\nabla P_t F|^2}{P_tF} \right]  .
\end{equation}
Let $\Phi(u) = u\log u$ ($u>0$). 
Then:
\begin{equation}
\frac{d}{dt}  \E_{\mu}\Big[ \Phi(P_t F) \Big]           
= 
- \E_{\mu}\Big[\frac{|\nabla P_t F|^2}{P_tF}\Big]
,
\label{eq_time_der_entropy}
\end{equation}
and the log-Sobolev inequality for the measure $\mu$  implies
($F$ has average $1$ under $\mu$):
\begin{equation}
\E_{\mu} \big[ \Phi(P_t F) \big]  
\leq e^{-2\gamma_{\rm LS}\, t} \; \E_{\mu} \big[ \Phi(F)  \big]  
.
\end{equation}
For any $\epsilon<2\gamma_{\rm LS}$, 
an integration by parts in time in~\eqref{eq_cov_continuous_0} therefore yields:
\begin{align}
\sum_{i=1}^N\cov_\mu(F,H_i)^2
&\leq 
\frac{1}{\epsilon^2}\Big[-\epsilon e^{\epsilon t}  \, \E_{\mu} \big[ \Phi(P_t F) \big]   \Big]_{0}^\infty 
+ 
\frac{1}{\epsilon^2}\int_0^\infty \epsilon^2 e^{\epsilon t}  \, \E_{\mu} \big[ \Phi(P_t F) \big]  
\, dt\nonumber\\
&= 
\frac{1}{\epsilon} \, \E_{\mu} \big[ \Phi(F) \big]  
+
\int_0^\infty e^{\epsilon t} \, \E_{\mu} \big[ \Phi(P_t F) \big]   \, dt
\leq
\Big(\frac{1}{\epsilon} 
+ \frac{1}{2\gamma_{\rm LS}-\epsilon}\Big) \, \E_{\mu} \big[ \Phi(F) \big]  
.
\end{align}
The right-hand side is minimal when $\epsilon=\gamma_{\rm LS}$, 
in which case it equals $2(\gamma_{\rm LS})^{-1} \, \E_{\mu} \big[   \Phi(F) \big]$.
Applying the log-Sobolev inequality yields the desired estimate:
\begin{equation}
\sum_{i=1}^N\cov_\mu(F,H_i)^2
\leq 
\frac{4}{\gamma_{\rm LS}^2} \max_{1\leq i\leq N}\|H'_i\|^2_\infty \E_\mu \big[\, |\nabla\sqrt{F}|^2\, \big]
.
\end{equation}
\end{proof}

\section{XY model}
\label{append: XY}

We consider the mean-field   XY model defined on the periodic compact space 
$(x_1, \dots , x_N) \in [0,2 \pi)^N$ with $V=0, W = - W^-$ and  
\begin{equation}
W^-(x,y) = \cos(x-y) = \cos(x)\cos(y) + \sin(x)\sin(y).
\label{eq: XY W}
\end{equation}
In this case, the Hilbert space $\bbH_0$ \eqref{eq_def_Hilbert_space} reduces to 
variables $\psi = (\zeta_1, \zeta_2)$ associated with the
2 modes $n_1(x) = \cos(x), n_2 (x) = \sin(x)$.
Note also that $\alpha_V$ is simply the uniform measure on $[0,2 \pi]$ as $V=0$.

We will check that  strong convexity of the renormalised potential, and thus a uniform log-Sobolev inequality, hold up to the critical temperature $T_c = 1/2$ (see \cite{Kirkpatrick-Nawaz} for the analysis of the equilibrium phase transition). 
This statement was already derived in~\cite{MR3926125} (under the name $O(2)$-model) and we recall the proof below for the sake of completeness.
 
\medskip
 
As $W^+ =0$, $U^N_T$ defined in \eqref{eq_def_renormalised_potential_nonquad} is independent of $N$ and given by 
\begin{align}
U^N_T(\zeta_1, \zeta_2)
&=
- \log \E_{\alpha_V}\bigg [  \exp \Big(  \frac{\zeta_1}{T}  \cos(x) + \frac{\zeta_2}{T} \sin(x)  \Big) \bigg] + {\rm constant}.
\end{align}
Thus the renormalised potential defined in \eqref{eq: multi-mode counterpart} reads
\begin{align}
\mathcal V_T (\zeta_1, \zeta_2)
 =  \frac{( \zeta_1^2 + \zeta_2^2)}{2 T} 
 - \log \int_0^{2 \pi} dx  \exp \Big(  \frac{\zeta_1}{T}  \cos(x) + \frac{\zeta_2}{T} \sin(x)  \Big) + {\rm constant}.
\end{align}
For any vector $v = (v_1,v_2)$, the quadratic form associated with the Hessian is given by    
\begin{align}
\big( v , \He \mathcal V_T  \, v \big)
 =  \frac{|v|^2}{T} 
- \frac{1}{T^2} \var_{\mu^{(\zeta_1, \zeta_2)}_T} \Big[  (v, 
\begin{pmatrix}
   \cos(x)\\
   \sin(x)
\end{pmatrix} ) \Big]
\geq \big( \frac{1}{T} - \frac{1}{2 \, T^2} \big) |v|^2, 
\label{eq: hess XY}
\end{align}
where $\mu^{(\zeta_1, \zeta_2)}_T (x) \propto \exp \Big(  \frac{\zeta_1}{T}  \cos(x) + \frac{\zeta_2}{T} \sin(x)  \Big)$ and the last inequality comes from the uniform upper bound by 1/2  on the variance in \eqref{eq: hess XY} established in \cite[Theorem D.2]{MR0496246}.
 
 As a consequence for any $T>T_c = 1/2$, the renormalised potential $\mathcal V_T$ is strongly convex and  the log-Sobolev inequality holds for the XY model by Theorem \ref{theo_nonquadratic}. 
By Lemma~\ref{lemma_cvx_cF_is_cvx_cV}, this implies the uniform convexity of $\hat  \cF_T$ \eqref{eq: projection F_T} up to $T_c$.

\section*{Acknowledgements}

We thank Songbo Wang for many useful discussions and sending us an early draft of his results. 

R.B. acknowledges funding from NSF grant DMS-2348045.

B.D. acknowledges funding from the European Union's Horizon 2020 research and innovation programme under the Marie Sk\l odowska-Curie grant agreement No 101034255.

\bibliography{all}
\bibliographystyle{plain}

\end{document}